\documentclass[11pt, reqno]{amsart}
\usepackage{amsthm,amssymb}
\usepackage{graphicx}
\usepackage{enumerate}
\usepackage{amssymb,amsmath,graphicx,amsfonts,euscript}
\pdfoutput=1
\usepackage{color}
\usepackage{extarrows}

\usepackage[margin=3cm, a4paper]{geometry}

 \def\norm#1{\|#1\|}
\def\normb#1{\bigg\|#1\bigg\|} \def\normo#1{\left\|#1\right\|}
\def\aabs#1{\left|#1\right|}
\def\abs#1{\left|#1\right|}
\def\H{{\cal H}}

\def\H1{H^1(\R)}

 \newcommand{\C}{\mathbb{C}}
 \newcommand{\R}{\mathbb{R}}
\newcommand{\Z}{\mathbb{Z}}

 \newcommand{\de}{\delta}
\newcommand{\e}{\varepsilon} 
 \newcommand{\la}{\lambda}
 \newcommand{\s}{\sigma}
  \newcommand{\x}{\xi}
\newcommand{\y}{\eta}

\newcommand{\De}{\Delta}

\newcommand{\p}{\partial} 
\newcommand{\re}{\mathop{\mathrm{Re}}}
\newcommand{\im}{\mathop{\mathrm{Im}}}

 \newcommand{\Del}[1]{}

\numberwithin{equation}{section}

\newtheorem{thm}{Theorem}[section]
\newtheorem{cor}[thm]{Corollary}
\newtheorem{lem}[thm]{Lemma}
\newtheorem{lemma}[thm]{Lemma}
\newtheorem{prop}[thm]{Proposition}
\newtheorem{definition}[thm]{Definition}
\newtheorem{remark}{Remark}

\theoremstyle{remark}

\newtheorem*{exam*}{Examples}

\begin{document}

\setcounter{page}{1}

\title[]{Well-posedness of the discrete nonlinear Schr\"odinger equations and the Klein-Gordon equations}

%

\author[Yifei Wu]{Yifei Wu}
\address{Center for Applied Mathematics, Tianjin University,
Tianjin 300072, P.R.China}
\email{yerfmath@gmail.com}
\thanks{Y.Wu and Z.Yang were in part supported by NSFC (Grant No. 12171356).}

\author[Zhibo Yang]{Zhibo Yang}
\address{Center for Applied Mathematics, Tianjin University,
Tianjin 300072, P.R.China}
\email{zbyang@tju.edu.cn}

\author{Qi Zhou}
\address{
Chern Institute of Mathematics and LPMC, Nankai University, Tianjin 300071, China
}
\thanks{Q.Zhou was supported by National Key R\&D Program of China (2020 YFA0713300), NSFC grant (12071232) and Nankai Zhide Foundation}
\email{qizhou@nankai.edu.cn}



\maketitle

\begin{abstract}

The primary objective of this paper is to investigate the well-posedness theories associated with the discrete nonlinear Schr\"odinger and Klein-Gordon equations. These theories encompass both local and global well-posedness, as well as the existence of blowing-up solutions for large and irregular initial data.

The main results of this paper presented in this paper can be summarized as follows:

1. Discrete Nonlinear Schr\"odinger Equation: We establish global well-posedness in $l^p_h$ spaces for all $1\leq p\leq \infty$, regardless of whether it is in the defocusing or focusing cases.

2. Discrete Klein-Gordon Equation (including Wave Equation): We demonstrate local well-posedness in $l^p_h$ spaces for all $1\leq p\leq \infty$. Furthermore, in the defocusing case, we establish global well-posedness in $l^p_h$ spaces for any $2\leq p\leq 2\s+2$. In contrast, in the focusing case, we show that solutions with negative energy blow up within a finite time.

These conclusions reveal the distinct  dynamic behaviors exhibited by the solutions of the equations in discrete settings compared to their continuous setting. Additionally, they illuminate the significant role that discretization plays in preventing ill-posedness and collapse phenomena.
\end{abstract}

\section{Introduction}
\subsection{Well-posedness theory of discrete nonlinear Schr\"odinger equation}

We consider the following discrete nonlinear Schr\"odinger equation (DNLS), 
\begin{align}\label{defDL}
\begin{split}
 \left \{
\begin{array}{ll}
iu^{\prime}_{n}(t)-\Delta_hu_n+V_n u_n+\lambda |u_n|^{2\s}u_n =0\\
{u_n(0)}=u_{n,0},
\end{array}
\right.
\end{split}
\end{align}
where $u=\{u_n\}_{n\in h\Z^d}:\R\times h\Z^d\rightarrow \C$ is complex-valued, $u_0=\{u_{n,0}\}_{n\in h\Z^d}$ is initial data, and $h>0$ denotes the stepsize of the lattice $h\Z^d$.
Here, we usually take $\lambda=\pm1$, while $\lambda=-1$  is called \textit{focusing}, and \textit{defocusing} for $\lambda=1$.  The corresponding discrete Schr\"odinger operator  takes the form 
\begin{equation}\label{sch}
H= -\Delta_hu_n+V_n u_n
\end{equation}
where  
\begin{align}\label{defLAP}
\Delta_h u_n = \sum_{j=1}^d \frac{u_{n+he_j}+u_{n-he_j}-2u_n}{h^2}
\end{align}
denotes the discrete Laplace operator for any $n\in h\Z^d$ with the canonical basis $(e_j)_{1\leq j\leq d} $ on $\R^d$, and  $V=\{V_n\}_{n\in h\Z^d}$ is a bounded real-valued potential.  What's of particular interest and receive wide study \cite{1993Localization,2008Theabs,2015GlobalThe,2017SharpPha,1989Anewp,1983Absen, 1999Metal,2018Univer,2000Schro,2023Exact} is the case 
\begin{itemize}
\item $V_n=0$ or $V_n$ is a periodic sequence.
\item $V_n$  is random, i,e.   it is  a family of independent identically distributed random variables on $[0,1]$.
\item $V_n$ is quasiperiodic, i.e, $V_n=f(\theta+n \alpha)$, where $f$ is a continuous function on $\mathbb{R}^d / \mathbb{Z}^d$, and $(1,\alpha)$ is  rationally independent.       
\end{itemize}

The DNLS is a fundamental mathematical model in physics with a wide range of applications. For example, it was widely used in the study of  one-dimensional arrays of coupled optical waveguides \cite{Eisenberg2002Optical}, the propagation of optical waves in nonlinear media \cite{2002Discrete, Sukhorukov2003Spatial}.

In the past decades, there has been a significant interest in finding special solutions of \eqref{defDL}. Examples include ground states \cite{2016Onsite, 2023Ground, 1999Excitation}, standing wave solutions with exponentially decaying amplitudes \cite{2006Gap, 2008Periodic}, solitary traveling waves \cite{2019Existence}, solutions that exhibit spatial localization and quasi-periodic behavior in time \cite{2008Quasi, 2014Localization, 2022Nonlinear}, as well as long-time Anderson localization \cite{2023Anderson, 2009LongTime}. And some results provide some estimates of the growth of discrete sobolev norms of the solution \cite{bernier2018bounds}. Of course, another fundamental challenge in comprehending partial differential equations lies in the theory of well-posedness.   However, as mentioned in \cite{MR2682116},  
well-posedness theory of \eqref{defDL} is not quite satisfactory.    Before explaining the results, we recall the following 
 standard definitions:

\begin{definition}[Well-posedness]\label{def1}
The well-posedness, blow-up criterion, and global well-posedness can be defined as follows:
\begin{enumerate}
\item We denote by $C_t\left(I;X_0\right)$ the space of continuous functions from time interval $I$ to the topological space $X_0$.We say that the Cauchy problem is locally well-posed in $C_t\left(I;X_0\right)$ if the following properties hold:
\begin{enumerate}
\item There is unconditional uniqueness in $C_t\left(I;X_0\right)$ for the problem.
\item For every $u_{0}\in X_0$, there exists a strong solution defined on a maximal time interval $I=\left(-T_{\min},T_{\max}\right)$, with $-T_{\min},T_{\max}\in (0,+\infty]$.
\item The solution depends continuously on the initial value.
\end{enumerate}

\item There is a blow-up alternative. If $T_{\max}<\infty$, then $\lim\limits_{t\rightarrow T_{\max}}\norm{u(t)}_{X_0}=+\infty$ (respectively, if $T_{\min}<\infty$, then $\lim\limits_{t\rightarrow -T_{\min}}\norm{u(t)}_{X_0}=+\infty$). In this case, we call the solution blows up in finite time.

\item If the maximal lifespan $I=\R$, then we call it globally well-posed.
\end{enumerate}
\end{definition}

It is worth noting that the definition above is referred to as the ``unconditional" well-posedness, which is stronger than the normal concept.

As of now, the global well-posedness theory of \eqref{defDL} has been primarily confined to weighted $l^2$-spaces \cite{bernier2018bounds, MR2682116, MR2147475}. It remains uncertain whether the solution to \eqref{defDL} remains well-posed in $l^p$ spaces where $p\neq 2$. Furthermore, when we deviate from the case where $V_n=0$ and introduce different potentials $V_n$, it significantly alters the distinctive properties of \eqref{sch}. For instance, when $V_n$ is random, \eqref{sch} typically exhibits a pure point spectrum \cite{1993Localization, 1983Absen}. However, if $V_n$ is quasi-periodic, it leads to various spectral behaviors such as pure point, absolutely continuous, and singular continuous spectrum \cite{2008Theabs, 2015GlobalThe, 2017SharpPha, 1999Metal, 2018Univer}. It is not yet clear whether this has an impact on the well-posedness problem. In this paper, we aim to address these questions. To begin, we will state the local well-posedness as follows:

\begin{thm}\label{thmLOC} Let $1\le p\leq+\infty$ and $\lambda =\pm 1$.
 Then the Cauchy problems \eqref{defDL}  are locally well-posed in $l^p_h$\footnote{See for section 2.1 for the definition of $l^p_h$.}.  
\end{thm}

The results in Theorem \ref{thmLOC} exhibit significant differences from the local well-posedness results of the continuous nonlinear Schr\"odinger equation (NLS):
\begin{equation}\label{def:NLS}
\left\{
\begin{aligned}
& iu^{\prime}(t)-\Delta u+V(x)u+\lambda |u|^{p}u =0, \\
& u(0)=u_{0},
\end{aligned}
\right.
\end{equation}
with $\lambda=\pm 1$, $p>0$. Here $u(t,x):\R\times \R^d\rightarrow \mathbb{C}$ is a complex-valued function.

It has been established that the Cauchy problem \eqref{def:NLS} with $V=0$ is well-posed for general initial data in the space $H^s(\R^d), s\ge s_c$ with $s_c=\frac{d}{2}-\frac{2}{p}$, as demonstrated by Cazenave and Weissler \cite{1988The}. 
Moreover, Christ, Colliander, and Tao \cite{2003Ill} showed the existence of initial data in $H^s(\R^d)$, where $s<s_c$, leading to ill-posedness in the Cauchy problem of NLS \eqref{def:NLS}. Furthermore, H\"ormander \cite{MR0121655} established ill-posedness in the context of $L^p$ spaces for any $p\neq 2$ . For more insights, please refer to recent research presented in \cite{DSS2020}. Nevertheless, our findings indicate that for DNLS, local well-posedness remains valid across all dimensions $d\geq 1$ and for any $p\geq 1$. 
Additionally, in order to ensure the well-posedness of our approach, our assumption concerning the potential is relatively lenient, requiring only boundedness. In contrast, within the context of the continuous version,  the potential assumption is notably more stringent. Specifically, it is often imperative to incorporate decay assumptions on the function $V(x)$ to guarantee favorable properties of the spectrum of linear operators. Related results, one may refer to \cite{1991Decay} and the references therein.  
These suggest that solutions of the discrete equation may exhibit greater stability compared to solutions of the continuous equation. The forthcoming global well-posedness result will further corroborate this observation, as follows:

\begin{thm}\label{thm1.3} Let $1\le p\leq+\infty$ and $\lambda =\pm 1$,   then the Cauchy problem \eqref{defDL} is globally well-posed in $l^p_h$. Moreover, the following inequality holds:
\begin{align}\label{Bouned-NLS-lp}
\norm{u(t)}_{l^p_h}\leq e^{2d|t|/h^2}\norm{u_{0}}_{l^p_h}, \quad \mbox{ for any } t\in \R.
\end{align}
\end{thm}

\begin{remark}
The bound \eqref{Bouned-NLS-lp} is clearly not sharp, since it is uniformly bounded in $l^2$. We conjecture that 
a finer estimate should be 
\begin{align*}
\norm{u(t)}_{l^p_h}\leq e^{2d|1-\frac 2p| h^{-2}|t|}\norm{u_{0}}_{l^p_h}, \quad \mbox{ for any } t\in \R,
\end{align*}
which matches the linear estimate presented in Lemma \ref{lem2.1}. 
\end{remark}

This theorem indicates that the long-time behavior of solutions in DNLS significantly deviates from that of NLS. It is widely recognized that, in the focusing case of NLS, there are solutions that undergo blow-up as $p$ surpasses the mass-critical power of $\frac{4}{d}$, as discussed in \cite{DWZ2016, 1977Glassey} and related references. This collapse is attributed to the presence of small wavelengths and large frequencies. For instance, consider the continuous NLS \eqref{def:NLS} with $p=\frac{4}{d}$. Research has established that solutions to \eqref{def:NLS} exhibit global existence when $\|u_0\|_{L^2_x}<2\pi$, while blow-up solutions can arise when $\|u_0\|_{L^2}\geq 2\pi$, as indicated in \cite{Caz03book}. However, our findings reveal that for discrete equations, global well-posedness remains valid across any dimension $d\geq 1$ and for any $p\geq 1$.

Numerous studies in the literature \cite{MR1199033, MR2983017, MR0762363} have also demonstrated that as the lattice step size $h$ approaches zero, equation \eqref{defDL} converges to equation \eqref{def:NLS}. The continuum limit of DNLS is a pivotal subject in theoretical research, and this matter has been extensively explored, see for examples \cite{YH2019, KLS2013}. As a result, our findings also take into account the relationship between the results and the step size $h$, even though prior research frequently employed $h=1$ as the standard setting.

\subsection{Well-posedness theory of discrete nonlinear Klein-Gordon equation}
Next, we turn our attention to the discrete nonlinear Klein-Gordon (DKG) equation which takes the following form:
\begin{align}\label{defKG}
\begin{split}
\left \{
\begin{array}{ll}
\partial_t^2 u_n(t)-\Delta_h u_n+V_n u_n+\lambda |u_n|^{2\sigma}u_n =0\\
u_n(0)=f_n,\quad \partial_t u_n(0)=g_n,
\end{array}
\right.
\end{split}
\end{align}
where $u=\{u_n\}_{n\in h\Z^d}:\R\times h\Z^d\rightarrow \R$ is real-valued,  $\s > 0, \lambda=\pm1$, and similar as above, $V=\{V_n\}_{n\in h\Z^d}$ is a bounded real-valued potential. In the following, we denote $(f,g)=\{(f_n,g_n)\}_{n\in h\Z^d}$. This equation is referred to as ``defocusing" when $\lambda=1$ and ``focusing" when $\lambda=-1$. When $V_n$ is near origin, the equation can be regarded as the discrete nonlinear wave equation with potential, which is also a subject of investigation in this paper.

The DKG has broad applications in physics, describing the behavior of fields and particles in discrete systems such as lattices \cite{M1989Statistical,1993Fluxon}. Furthermore, the equation holds particular interest due to its relativistic invariance and can be considered as the relativistic counterpart of the Schr\"odinger equation, as discussed in \cite{lei2023nonrelativistic}.

In contrast to the DNLS, the well-posedness theory of the DKG is considerably more limited. To the best of our knowledge, only small data global well-posedness results in the space $l^2\times l^2$ for dimensions up to four and under certain restrictions on the parameter $p$ are known, as shown in \cite{2015Dispersion, 2005Asymptotic}.

In our paper, we will establish well-posedness results for more general and larger initial data for the DKG. Our first theorem addresses the local well-posedness.:
\begin{thm}\label{thmLOCKG}
Let $1\le p \leq+\infty,0<h\leq 1$.
Suppose that $(f,g)\in l^p_h\times l^p_h$, then the Cauchy problem \eqref{defKG} is locally well-posed in $l^p_h\times l^p_h$.
\end{thm}

In order to draw a comparison with the impact of discretization in DNLS, we also present several well-posedness results for the continuous version of the Klein-Gordon equation (KG). The continuous Klein-Gordon equation is defined as follows:
\begin{align}\label{def:KG}
\begin{split}
\left \{
\begin{array}{ll}
\partial_t^2 u(t)-\Delta u+m u+\lambda |u|^pu =0\\
u(0)=u_0,\quad \partial_t u(0)=u_1,
\end{array}
\right.
\end{split}
\end{align}
where $u(t,x):\R\times \R^d\rightarrow \R$, the constants $m\ge 0$ and $\lambda=\pm 1$. 

Similar to the NLS, one should not expect to solve the Cauchy problem \eqref{def:KG} in $L^p$-based Sobolev spaces when $p\neq 2$. However, in the DNLS, we observe that the breakdown of the linear flow from $l^p_h$ to $l^p_h$ can be alleviated through discretization. This enables us to establish local well-posedness in the space $l^p_h\times l^p_h$ for any $1\leq p\leq +\infty$.

Regarding global well-posedness, one might assume that the same mechanism for achieving global results, as seen in the DNLS, could be applicable to the DKG equation. However, we've discovered that the role of discretization in the DKG equation differs from that in the DNLS. In the case of DNLS, discretization averts solution blow-up. Nonetheless, for DNLS, we have managed to establish more comprehensive well-posedness results in the defocusing scenario. However, in the focusing scenario, blow-up solutions continue to exist. 
Our specific results are as follows:

In the defocusing case, we proceed to establish the global well-posedness of the equation in the space $l^p_h\times l^p_h$ for any $2\leq p\leq 2\s+2$. The theorem is presented as follows:
\begin{thm}\label{thm1.4}
 Let $\s\geq 0,0<h\leq 1$,  $\lambda = 1$, and $2\leq p\leq 2\s+2$. Moreover, fix $\delta_0>0$ and assume that
\begin{align}\label{assump-Vn-KG-1}
\inf_{n\in h\Z^d}\big(h^2 V_n+2d)> 0.
\end{align}
Suppose that $(f,g)\in l^p_h\times l^p_h$, then the Cauchy problem  \eqref{defKG} is globally well-posed in $(f,g)\in l^p_h\times l^p_h$.  Moreover, 
there holds 
\begin{align*}
\norm{\big(u(t),\partial_t u(t)\big)}_{l^p_h\times l^p_h}\leq e^{Ch^{-1}|t|} \norm{\big(f,g\big)}_{l^p_h\times l^p_h}, \quad \mbox{ for any } t\in \R,
\end{align*}
where the constant $C>0$ is not dependent on $t,h$ and $(f,g)$.
\end{thm}

Furthermore, this theorem marks a significant advancement in our comprehension of the well-posedness of the discrete equation. In contrast to prior findings restricted by dimensions and the size of the initial data, our theorem is applicable to any dimension $d\geq 1$ and general initial values in $l^p_h$. Moreover, it eases the constraints on the potential, which only needs to satisfy \eqref{assump-Vn-KG-1}, encompassing a broader range of scenarios.

We conjecture that a similar conclusion holds for $p>2\sigma+2$. However, our current method fails as it relies on the basic energy estimate. 
For the focusing case, define the erergy
\begin{align*}
E\left(u,\p_tu\right)=\frac{1}{2}\sum_{n\in h\Z^d}\left(\aabs{\p_t u_n}^2+\frac{1}{h^2}\sum_{j=1}^d\left(u_{n+he_j}-u_n\right)^2+V_n|u_n|^2+\frac{\lambda}{\s +1}|u_n|^{2\s+2}\right),
\end{align*}
which is conserved under the nonlinear flow \eqref{defKG}.

Then our conclusion is as follows:
 \begin{thm}\label{mainthm3}
Let $d\ge 1, \lambda = -1$. Moreover, assume that  
$$
\inf_{n\in h\Z^d} V_n> 0.
$$
 Suppose that $(f,g)\in l^2_h\times l^2_h$ with $E(f,g)<0$, then the solution to the Cauchy problem  \eqref{defKG} with initial data  $(f,g)$ blows up at time $T_*<\infty$. Moreover,
 $$
 \lim\limits_{t\to T_*} \|u_n(t)\|_{l^2_h}=+\infty.
 $$  
 \end{thm}

\begin{remark}
There exist a class of pairs $(f,g)$ with $E(f,g)<0$. Taking an example,  choose $g_n= 0$ and 
\begin{align*}
\begin{split}
f_n= \left \{
\begin{array}{ll}
(V_n\s+V_n)^{\frac{1}{\s}},\quad n = \underbrace{[0,\cdots ,0]}_d, \\
0,\quad n\neq \underbrace{[0,\cdots ,0]}_d,
\end{array}
\right.
\end{split}
\end{align*}
then $E(f,g)<0$.
\end{remark}

Although we were not able to establish global well-posedness results for DKG that are applicable to all values of $p$ ranging from $1$ to $+\infty$, we discovered that the global well-posedness of DKG closely mirrors that of its continuous counterpart, the KG equation. In the continuous version, for cases where $p\leq \frac{4}{d-2}$, the global dynamics of the solutions have been thoroughly explored by numerous researchers. Specifically, it has been proven to exhibit global well-posedness in the energy space $H^1(\R^d)\times L^2(\R^d)$ for defocusing cases. However, in focusing cases, it has been shown that arbitrary initial data do not lead to global well-posedness. Instead, solutions below and above the ground energy threshold bifurcate into globally well-posed solutions and blow-up solutions. You can refer to relevant results in \cite{GV1985, 1989The, 1999Scattering, 2011Invariant, 1975Saddle, 1993The}.

\subsection{Novelty, ideas of proof}

In the following, we will elucidate the critical components of the proofs underlying our main theorems.

$\bullet$ Linear estimates  for the discrete Schr\"odinger and Klein-Gordon flows.

The key element in proving the well-posedness of DNLS and DKG is the establishment of new estimates for the linear operators $e^{-it\Delta_h}$ and $e^{-it\sqrt{1-\Delta_h}}$, which can provide boundedness estimates from $l^p_h$ to $l^p_h$ spaces. For the Schr\"odinger flow, we  obtain, for any $1\leq p\leq +\infty$, the following estimate:
\begin{align}\label{lp-bound_1}
\norm{e^{-i\De_h t}\phi}_{l^p_h}\le e^{C|t|}\norm{\phi}_{l^p_h},\quad \phi\in l^p_h(h\Z^d).
\end{align}
Here, $C>0$ is a constant that depends only on the dimensions $d$ and $h$. Especially,  this estimate holds uniformly $p$, which allows us to include the case $p=+\infty$. 
Similarly, for the Klein-Gordon flow, the estimate is as follows, for any $1\leq p\leq +\infty$:
\begin{align}\label{lp-bound}
\norm{e^{-it\sqrt{1-\Delta_h}}\phi}_{l^p_h} \leq e^{C|t|}\norm{\phi}_{l^p_h}, \quad \phi\in l^p_h(h\Z^d).
\end{align}

With the linear operator estimates, the results of local well-posedness in $l^p_h$ spaces can be directly obtained through the standard fixed-point method. It's important to note that these estimates are established using a direct and elementary approach, without relying on the spectral analysis of the linear operator. This is the reason we have addressed certain limitations in classical well-posedness research methods.

In the traditional research approaches for well-posedness of equations, dispersion estimates play a pivotal role. Proving such estimates often comes down to establishing decay estimates for the $L^\infty$ ($l^\infty_h$ in the discrete case) norm of the solution at time $t$ in relation to the $L^1$ ($l^1_h$ in the discrete case) norm of its initial data. For instance, in the continuous Schr\"odinger equation, it's well-known that for any $t\in \R\setminus{0}$, the following decay estimate holds:
\begin{align}\label{decay-LS-conti}
\norm{e^{-it\De}\phi}_{L_x^{\infty}}\leq C|t|^{-d/2}\norm{\phi}_{L_x^1},\quad \phi\in L^1(\R^d).
\end{align}
Well established $L^1\rightarrow L^\infty$  decay estimate give rise to a whole family of mixed space-time norm estimates, called Strichartz estimates \cite{ 1992Smoothing, 2019endpoint, Strichartz1977Restrictions}.  Strichartz estimates can be used in conjunction with a contraction mapping argument to prove global well-posedness for certain nonlinear equations with small initial data by a standard fixed point method.

In the discrete case with $V_n=0$, Stefanov and Kevrekidis \cite{2005Asymptotic} proved that:
\begin{align*}
\norm{e^{-it\De_h}\phi}_{l^\infty_h}\lesssim \left\langle t\right\rangle^{-d/3} \norm{\phi}_{l^1_h},\quad \phi\in l^1_h(\Z^d).
\end{align*}
This estimate, while strictly weaker than its continuous version \eqref{decay-LS-conti}, is considered sharp \cite{2005Asymptotic}. 
For DNLS with non-zero potential,  Pelinovsky and Stefanov \cite{2008PAOn} proved that
\begin{align}\label{desch}
\norm{e^{-it H} P_{ac}\phi}_{l^\infty_h}\lesssim \left\langle t\right\rangle^{-1/3}\norm{\phi}_{l^1_h},\quad \phi\in l^1_h(\Z),
\end{align}
where the ``generic" potentials $V_n$ decay sufficiently rapidly at infinity. Here, $P_{ac}$ represents the projection onto the absolutely continuous part of the spectrum. More recently, if $d=1$ and $V_n$ is quasiperiodic and analytically small enough (resulting in the Schr\"odinger operator having a purely absolutely continuous spectrum), Bambusi and Zhao \cite{MR4070305} obtained an estimate similar to \eqref{desch}. For a deeper exploration of related research, one may refer to \cite{Bambusi2013Asymptotic, 2008On, 2015Dispersion, MR2578796, 2006Dispersive}, and the references therein. The estimates mentioned above heavily depend on the assumption that the linear Schr\"odinger operator has an absolutely continuous spectrum (obviously $\De_h$ has absolutely continuous spectrum). However, if $V_n$ is random and $d=1$, the operator consistently possesses a pure point spectrum \cite{1993Localization, 1989Anewp, 1983Absen}. In cases where $d\geq 2$, it remains an open question whether the operator has an absolutely continuous part \cite{2000Schro}. If $V_n$ is almost periodic, singular continuous spectrum \cite{2015GlobalThe, 2017SharpPha}, or a mobility edge (energy level that separates absolutely continuous spectrum and pure point spectrum) \cite{2023Exact} may emerge, and thus it remains an open question whether the corresponding Schr\"odinger operator possesses Strichartz estimates.

Considerable research has also been conducted on the DKG equation. Decay estimates for this equation were originally formulated by Stefanov and Kevrekidis \cite{2005Asymptotic} in one-dimensional cases and were subsequently extended to higher dimensions ($d\ge 4$) by Cuenin and Ikromov \cite{2020Sharp}.  Similar to the DNLS, the decay rates in these estimates are weaker than those in the continuous version. There are only a few related results on the DKG with a potential, with one-dimensional estimates available in \cite{2015Dispersion}.

Certainly, owing to the significant reliance of Strichartz estimates on the continuous spectra of operators, existing findings impose specific constraints concerning the dimension and the potential $V_n$. In our presented approach within this paper, we incorporate the potential $V_n$ as an integral part of the nonlinearity within the equation. Consequently, we only necessitate that $V_n$ be bounded, thereby facilitating the establishment of well-posedness for the solution through our linear operator estimates \eqref{lp-bound_1} and \eqref{lp-bound} by employing Picard iterations. This approach allows us to circumvent the limitations associated with operator spectra and attain more broadly applicable well-posedness outcomes.


$\bullet$ A priori estimate for the nonlinear discrete Schr\"odinger and Klein-Gordon flows.

When considering the long-time behavior of the solution, we draw inspiration from the following observation. If we neglect the linear component and focus solely on the nonlinear flow, we encounter the following nonlinear Schr\"odinger equation (in the zero-dimensional case):
\begin{equation}\label{non-LS}
 \left\{\begin{aligned}
& i\partial_tu
=\lambda|u|^{2\s}u,
 && \\
 &u(0)=\phi. &&
 \end{aligned}\right.
\end{equation}
Here, we introduce the notation $\mathcal N_t$ to represent the flow as follows:
$$
\mathcal N_t(\phi)=e^{-i\lambda t|\phi|^{2\s}}\phi,
$$
then $\mathcal N_t(\phi)$ solves the equation \eqref{non-LS}. Additionally, regardless of whether $\lambda=1$ or $\lambda=-1$, we observe that $\|\mathcal N_t(\phi)\|_{l^p_h}=\|\phi\|_{l^p_h}$, which makes it evident that the solution $\mathcal N_t(\phi)$ enjoys global existence. Given the relatively weak influence of the linear flow, we establish the same phenomenon for the original nonlinear equations \eqref{defDL}. In paticular, a crucial observation for the DNLS is that we have an a priori estimate in $l^p_h$ norms:
$$
\norm{u(t)}_{l^p_h}\leq e^{C|t|}\norm{u_{0}}_{l^p_h},\quad 1\le p\le +\infty.
$$
This estimate obviously fails for the continuous NLS.   
 
%
%
%
%
%

In contrast to DNLS, the problem for DKG is even more intricate. In the case of DNLS, we were able to leverage the mass conservation law and a specific mechanism to establish a priori estimates in $l^p_h$. However, for DKG, there is no mass conservation law, and there is no analogous mechanism, as in the case of DNLS, to achieve such a priori estimates in $l^p_h\times l^p_h$.

Boundedness in $l^2_h\times l^2_h$ can be directly obtained by employing an energy-like estimate, as demonstrated in Lemma \ref{lem2.2.3} below. To extend these estimates to $l^p_h\times l^p_h$ where $p\neq 2$, we apply a ``linear-nonlinear decomposition" to the solution. Specifically, we decompose $u$ into two components: $u = w + v$, where $(v,\partial_t v) = S(t)(f,g)$, and $S(t)$ represents the linear operator of the Klein-Gordon solution.
The component $w_n$ satisfies the following equation:
\begin{align*}
\partial_{tt}w_n - \Delta_hw_n + V_nw_n = -|u_n|^{2\s}u_n,
\end{align*}
with zero initial data: $(w(0),\partial_tw(0)) = (0,0)$. To estimate $v$ and $w$, we employ different approaches. For $v$, we use the linear estimate \eqref{lp-bound}. For $w_n$, given its trivial initial data, we obtain the $l^2_h$-estimate by applying modified energy estimates.

$\bullet$ Blowing-up for the focusing nonlinear discrete  Klein-Gordon equations.

Now we consider the blowing-up in focusing case, motivated by the following observation. As previously discussed in the DNLS context, we consider the DKG in the zero-dimensional case, represented as:
\begin{equation}\label{non-KG}
 \left\{\begin{aligned}
& \partial_{tt}u
=|u|^{2\s}u,
 && \\
 &u(0)=f, \quad \partial_t u(0)=g. &&
 \end{aligned}\right.
\end{equation}
We denote the flow $\mathcal N_t(f,g)$ to be the solution to equation (\ref{non-KG}). By examining, we can easily verify that: 
\begin{align*}
\left[\mathcal N_t(f,g)^{-\s}\right]^{\prime\prime}=(\s+1)\mathcal N_t(f,g)^{-\s-2}\left(g^2-\frac{f^{2\s+2}}{2\s+2}\right).
\end{align*}
This equation leads to a crucial observation:
\begin{align}\label{KG-BP-zero}
g^2< \frac{1}{2\s+2}f^{2\s+2}\quad  \Longrightarrow \quad
 [\mathcal N_t(f,g)^{-\s}]^{\prime\prime}<0.
\end{align}
This observation indicates that $\mathcal N_t(\phi)$ blows up in finite time, revealing a completely different dynamic from the DNLS.

Hence, the key to proving Theorem \ref{mainthm3} hinges on establishing the following inequality:
$$
\Big[\Big(\sum_{n\in h\Z^d}|u_n(t)|^2\Big)^{-\beta}\Big]^{\prime\prime}<0, 
$$
under suitable conditions on the initial data and by carefully choosing an appropriate value for $\beta$, as demonstrated in \eqref{KG-BP-zero}. 
In Section \ref{sec:blowup}, we will demonstrate that if the initial data satisfies $E(f,g) < 0$, selecting $\beta = \sigma/2$ will meet the necessary conditions for the proof. 

\vskip 1cm

\section{Preliminary}
\subsection{Notation}
Let $C>0$ denote some constant, and write $C(a)>0$ for some constant depending on coefficient $a$. If $f\leq C g$, we write $f\lesssim g$.

The $l^p_h$ norm is defined as
$$\norm{u}_{l^p_h}\triangleq \left(\sum_{n\in h\Z^d}|u_n|^p\right)^{1/p},$$
where $u=\{u_n\}_{n\in h\Z^d}$. For $p=+\infty$, we define
$$
\norm{u}_{l^\infty_h}\triangleq \sup\limits_{n\in h\Z^d}\{|u_n|\}.
$$

For $I\subset R$, we use space $L_t^q-_h$ with the norm
$$\norm{u(t)}_{L_t^ql^p_h(I)}\triangleq \normo{\norm{u(t)}_{l^p_h}}_{L_t^q(I)}.$$

We recall the definition of the \emph{discrete Fourier transform} of a function $g\in l^2_h(h\Z^d)$, namely
\begin{align}\label{def:dFT}
\hat{g}(\x)\triangleq\sum_{n\in h\Z^d}g_n e^{-in\cdot \x},
\end{align}
where $\x\in \R^d $, and that we have an inversion formula: for all $n\in h\Z^d$,
\begin{align}\label{def:dFT-1}
g_n =\frac{h^d}{(2\pi)^d}\int_{[0, 2\pi h^{-1}]^d}\hat{g}(\x)e^{in\cdot\x}d\x.
\end{align}
Denote $I_h=[0, 2\pi h^{-1}]^d$, and the inner products in discrete and continuous versions,   and norm of $L^2(I_h)$ are defined by 
\begin{align*}
\langle f,g\rangle 
\triangleq &
\int_{I_h} f(x)\overline{g(x)}\, d x, 
\quad\mbox{if }\quad f,g \in L^2(I_h);\\
\langle f,g\rangle 
\triangleq&
 \sum_{n\in h\Z^d} f_n \overline{g_n},  
\quad\qquad\mbox{if }\quad f,g \in l^2_h;\\
\|f\|_{L^2(I_h)}
\triangleq  & \sqrt{\langle f,f\rangle} .
\end{align*}
Then the following standard properties of the Fourier transform are well known:  
\begin{align*}
\|f\|_{l^2_h} 
&= (2\pi h^{-1})^{-\frac d2}\big\|\hat f\big\|_{L^2(I_h)} && \mbox{(Plancherel's identity)}  \\
\sum\limits_{n\in h\Z^d} f_n \overline{g_n}  &  =(2\pi h^{-1})^{-d} \langle \hat f,\hat g\rangle  && \mbox{(Parseval's identity)}.
\end{align*}


\subsection{Discrete multiplier estimate}

In this subection, we consider the estimates on the discrete pseudo-differential operators.  
The estimates presented here which may be of interest by itself.  In particular, we focus our attention on the stepsize dependence estimates on the following basic operators 
$$
(1-\Delta_h)^\alpha \quad \mbox{ and }\quad  (1-\Delta_h)^{-\alpha},\quad \mbox{for some } \alpha>0.
$$

First, we have that 
\begin{lemma}\label{lem2.0}
Let $\alpha\in [0,1]$, $0< h\leq 1$ and $p\in [1,+\infty]$, then 
\begin{align*}
\norm{\left(1-\De_h\right)^\alpha f}_{l^p_h}\le \big(4dh^{-2}\big)^\alpha \norm{f}_{l^p_h}.
\end{align*}
\end{lemma}
\begin{proof}
It suffices to show 
\begin{align*}
\norm{\left(1-\De_h\right)f}_{l^p_h}\leq \frac{1}{h^2}\norm{f}_{l^p_h}.
\end{align*}
Then the desired estimate is followed by the interpolation. 

From the definition \eqref{defLAP},  
\begin{align*}
&\norm{\left(1-\De_h\right)f}_{l^p_h} 
                                   \leq \norm{f}_{l^p_h}+ \norm{ \De_h f}_{l^p_h}.
\end{align*}
Note that, by norm inequality and the definition of $\De_h$
\begin{align*}
\norm{ \De_h f}_{l^p_h}
\leq
\frac{4d}{h^2}\norm{f}_{l^p_h}
\end{align*}
This gives the claimed estimate and thus finishes the proof of the lemma.   
\end{proof}

Next, we consider the operator $(1-\Delta_h)^{-\alpha}$. 
Firstly we prove the discrete  Mihlin-H\"ormander Multiplier Theorem. We emphasis that  the condition given below is different from the classical Mihlin-H\"ormander Multiplier Theorem, see \cite{2011Invariant}. However, due to the complex structure of the pseudo-differential operators in discrete version, our result below can be used to lower down the singularity of the stepsize. 

To do this, we define the operator $T$, 
$$
\widehat{Tf}(\x)=m(\x)\hat{f}(\x).
$$ 
Then we have 
\begin{lemma}\label{lem2.3.1}
Let $d\geq 1$, $1\leq p\leq+\infty$ and $0<h\leq 1$, and let $m(\x)$ be a nonzero smooth function on $I_h=[0,2\pi h^{-1}]^d$. For any fixed $N\ge 1$,  suppose that there exist some constant $A_0>0$ such that
\begin{equation}  \label{condition-m}
\begin{aligned}
\|m\|_{L^\infty(I_h)}
&+(\ln N)^d \big\|\p_{\x_1}\p_{\x_2}\cdots \p_{\x_d} m\big\|_{L^1(I_h)}\\
&+(hN)^{-1} (\ln N)^{d-1}\sum\limits_{j=1}^d \big\|\p_{\x_j}\p_{\x_1}\p_{\x_2}\cdots \p_{\x_d} m \big\|_{L^1(I_h)}\leq A_0.
\end{aligned}
\end{equation}
Then for any $f\in l^p_h$, 
$$
\big\|Tf\big\|_{l^p_h}\leq CA
_0\norm{f}_{l^p_h},
$$
where the constant $C>0$ is not dependent of $A_0$ and $f$.
\end{lemma}
\begin{proof}
Let $f=\{f_n\}_{n\in h\Z^d}$. By definitions \eqref{def:dFT} and \eqref{def:dFT-1}, we have that 
\begin{align*}
(Tf)_n
= &
\frac{h^d}{(2\pi)^d} \int_{[0,2\pi h^{-1}]^d}e^{in\cdot\x} m(\xi)\hat f(\xi)\,d\x\\
=&
 \sum_{n'\in h\Z^d} f_{n'}\frac{h^d}{(2\pi)^d}\int_{[0,2\pi h^{-1}]^d}  e^{-i(n^{\prime}-n)\cdot\x}  m(\xi)  \,d\x.
\end{align*}
Then we have that 
\begin{subequations}\label{eqlem2.3.1.1}
\begin{align}
\big\|Tf\big\|_{l^p_h}
&= 
\left(\sum_{n\in h\Z^d} \bigg|\sum_{n^{\prime}\in h\Z^d}f_{n^{\prime}}\frac{h^d}{(2\pi)^d}\int_{[0,2\pi h^{-1}]^d}e^{-i(n^{\prime}-n)\cdot\x}m(\x)\,d\x\bigg|^p\right)^{1/p}\notag\\
&= 
\left(\sum_{n\in h\Z^d}\bigg|\sum_{n^{\prime}\in h\Z^d}f_{{n^{\prime}+n}}\frac{h^d}{(2\pi)^d}\int_{[0,2\pi h^{-1}]^d}e^{-in^{\prime}\cdot\x}m(\x)\,d\x\bigg|^p\right)^{1/p}\notag\\
&= 
\normb{f\frac{h^d}{(2\pi)^d}\int_{[0,2\pi h^{-1}]^d} m(\x)\,d\x}_{l^p_h}\label{eqlem2.3.1.1-1}\\
&\quad +
\left(\sum_{n\in h\Z^d}\bigg|\sum_{n^{\prime}\in h\Z^d:n'\ne 0}f_{{n^{\prime}+n}}\frac{h^d}{(2\pi)^d}\int_{[0,2\pi h^{-1}]^d}e^{-in^{\prime}\cdot\x}m(\x)\,d\x\bigg|^p\right)^{1/p}.\label{eqlem2.3.1.1-2}
\end{align}
\end{subequations}
For the term \eqref{eqlem2.3.1.1-1}, there exist some absolute constant $C>0$ such that 
\begin{align*}
\eqref{eqlem2.3.1.1-1} 
&\leq 
\frac{h^d}{(2\pi)^d}\int_{[0,2\pi h^{-1}]^d}|m(\x)|\,d\x\> \big\|f\big\|_{l^p_h}\\
&\le 
C\|m\|_{L^\infty(I_h)} \big\|f\big\|_{l^p_h} \\
&\le 
CA_0 \big\|f\big\|_{l^p_h}.
\end{align*}
For the term \eqref{eqlem2.3.1.1-2}, we have that 
\begin{align}\label{est: eqlem2.3.1.1-2}
\eqref{eqlem2.3.1.1-2} 
&\leq 
\norm{f}_{l^p_h}\normb{\frac{h^d}{(2\pi)^d}\int_{[0,2\pi h^{-1}]^d}e^{-in\cdot\x}m(\x)\,d\x}_{l^1_h}.                                                                    
\end{align}
Denote $b=\{b_n\}_{n\in h\Z^d}$, where
$$
b_n\triangleq \frac{h^d}{(2\pi)^d}\int_{[0,2\pi h^{-1}]^d}e^{-in\cdot\x}m(\x)d\x.$$
and 
$\xi=(\xi_1,\xi_2,\cdots,\xi_d), n=(n_1,n_2,\cdots,n_d)$. 
Applying the formula
$$
e^{-in_j  \x_j}=\frac{1}{-in_j}\partial_{\xi_j}\Big( e^{-in_j \x_j}\Big),
$$
and integration-by-parts we get
\begin{align*}
b_n &= \left(\frac{h}{2\pi}\right)^{d}\int_{[0,2\pi h^{-1}]^d} e^{-i\sum_{j=1}^d n_j \x_j} m(\x) d\x_1d\x_2\cdots d\x_d \\
    &= \left(\frac{h}{2\pi}\right)^{d} \frac{i}{n_1}\int_{[0,2\pi h^{-1}]^{d-1}}e^{-in\cdot\x}m(\x) \>  d\x_2\cdots d\x_d \bigg|_{\x_1 = 0}^{\x_1=2\pi h^{-1}}\\
    &\quad 
    -\left(\frac{h}{2\pi}\right)^{d} \frac{i}{n_1}\int_{[0,2\pi h^{-1}]^d}e^{-in\cdot\x}\p_{\x_1}m(\x)\> d\x_1d\x_2\cdots d\x_d .
\end{align*}
Note that the first term vanishes by the periodicity, we get that 
\begin{align*}
b_n=-\left(\frac{h}{2\pi}\right)^{d} \frac{i}{n_1}\int_{[0,2\pi h^{-1}]^d}e^{-in\cdot\x}\p_{\x_1}m(\x)\> d\x_1d\x_2\cdots d\x_d.
\end{align*}
Then using the same process for $j=2,\cdots, d$,  we further obtain that 
\begin{align}\label{bn-form1}
b_n=\frac{h^{d}}{n_1n_2\cdots n_d}\left(\frac{1}{2i\pi}\right)^d\int_{[0,2\pi h^{-1}]^d}e^{-in\cdot\x}\p_{\x_1}\p_{\x_2}\cdots \p_{\x_d} m(\x) d\x_1d\x_2\cdots d\x_d.
\end{align}
Without loss generality, we may assume that $|n_1|\ge |n_2|\ge\cdots \geq|n_d|$.  
Then,  using the same process once again in $j=1$,  we finally get that 
\begin{align}\label{bn-form2}
b_n=\frac{h^{d}}{n_1^2 n_2\cdots n_d}\left(\frac{1}{2i\pi}\right)^{d+1}\int_{[0,2\pi h^{-1}]^d}\p_{\x_1}^2\p_{\x_2}\cdots \p_{\x_d} m(\x) e^{-in\cdot\x}d\x_1d\x_2\cdots d\x_d.
\end{align}
Fix a number $N$ which will be determined later. Then by \eqref{bn-form1},  we have that 
\begin{align}\label{bn-low}
\norm{b}_{l^1_h\{|n|\le hN\}}\leq &
C\Big\|\frac{h^d}{n_1 n_2\cdots n_d}\Big\|_{l^1_h\{|n|\le hN\}}
  \int_{[0,2\pi h^{-1}]^d}\big|\p_{\x_1}\p_{\x_2}\cdots \p_{\x_d} m(\x)\big| d\x_1d\x_2\cdots d\x_d\notag\\
\le &
C(\ln N)^d
  \big\|\p_{\x_1}\p_{\x_2}\cdots \p_{\x_d} m\big\|_{L^1(I_h)}.
\end{align}

If $|n|>hN$, then there exists an integer $j_0\in [1,d]$ such that 
$$
|n_1|\ge\cdots \ge |n_{j_0}|\ge hN\ge |n_{j_0+1}|\ge\cdots\ge  |n_d|.
$$
In this case, we use \eqref{bn-form2} and obtain that 
\begin{align*}
\norm{b}_{l^1_h\{|n|> hN\}}\leq &
Ch^{-1}\Big\|\frac{h^{d+1}}{n_1^2 n_2\cdots n_d}\Big\|_{l^1_h\{|n_1|\ge\cdots\ge |n_{j_0}|\ge hN\ge |n_{j_0+1}|\ge\cdots\ge  |n_d|\ne 0\}}\\
&\qquad \cdot \int_{[0,2\pi h^{-1}]^d}\big|\p_{\x_1}^2\p_{\x_2}\cdots \p_{\x_d} m(\x)\big| d\x_1d\x_2\cdots d\x_d.
\end{align*}
Since 
$$
\Big\|\frac{h^{d+1}}{n_1^2 n_2\cdots n_d}\Big\|_{l^1_h\{|n_1|\ge\cdots\ge |n_{j_0}|\ge hN\ge |n_{j_0+1}|\ge\cdots\ge  |n_d|\ne 0\}}
\le CN^{-1} (\ln N)^{d-j_0},
$$
we further get 
\begin{align*}
\norm{b}_{l^1_h\{|n|> hN\}}
\le &
C\sup\limits_{j_0\in [1,d]}(hN)^{-1} (\ln N)^{d-j_0} \> \int_{[0,2\pi h^{-1}]^d}\big|\p_{\x_1}^2\p_{\x_2}\cdots \p_{\x_d} m(\x)\big| d\x_1d\x_2\cdots d\x_d\\
\le &
C(hN)^{-1} (\ln N)^{d-1}\big\|\p_{\x_1}^2\p_{\x_2}\cdots \p_{\x_d} m\big\|_{L^1(I_h)}.
\end{align*}
Together this estimate with \eqref{bn-low}, we have that 
\begin{align*}
\norm{b}_{l^1_h}
\le &
C(\ln N)^d
  \big\|\p_{\x_1}\p_{\x_2}\cdots \p_{\x_d} m\big\|_{L^1(I_h)}
  +C(hN)^{-1} (\ln N)^{d-1}\big\|\p_{\x_1}^2\p_{\x_2}\cdots \p_{\x_d} m\big\|_{L^1(I_h)}\\
  \le &
  A_0.
\end{align*}
Therefore, inserting this estimate into \eqref{est: eqlem2.3.1.1-2}, we obtain that 
\begin{align*}
\eqref{eqlem2.3.1.1-2} 
&\leq 
CA_0\norm{f}_{l^p_h}.                                                                    
\end{align*}

Combining with the two estimates on \eqref{eqlem2.3.1.1}, we have
\begin{align*}
\big\|Tf\big\|_{l^p_h}
&\leq 
CA_0\norm{f}_{l^p_h},
\end{align*}
and thus  finish the proof of the lemma.
\end{proof}
\begin{remark}
It is worth noting that for the case of $p=2$, we only need $m$ to be boundedness in $L^\infty(\R^d)$.
\end{remark}

An application of the lemma above is the $l^p$-estimate for the operator $\left(1-\De_h\right)^{- \alpha}$. We expect that 
it is  bounded uniformly in stepsize from $l^p_h$ to $l^p_h$, however, what we can obtain in the following still has some log loss in $h$. 
\begin{cor}\label{cor1}
Let $\alpha\in [0,1]$ and $p\in [1,+\infty]$, then for any $f=\{f_n\}\in l^p_h$, 
\begin{align*}
 \norm{\left(1-\De_h\right)^{- \alpha}f}_{l^p_h}\leq C\big(1+|\ln h|\big)^{d\alpha}\norm{f}_{l^p_h},
\end{align*}
where the constant $C$ only dependent of $d$.
\end{cor}
\begin{proof}
It suffices to show 
\begin{align}\label{eqcor1.2}
\norm{\left(1-\De_h\right)^{- 1}f}_{l^p_h}\leq \frac{1}{h^2}\norm{f}_{l^p_h}.
\end{align}
Then the general case is followed by the interpolation. 

 We first drive the multiplier of $(1-\De_h)$. We denote $\xi=(\xi_1,\xi_2,\cdots,\xi_d)$, then 
\begin{align*}
\mathcal F\big[\De_h f\big](\xi) 
&=\frac{1}{h^2}\sum_{n\in h\Z^d}e^{-in\cdot\x}\sum_{j=1}^d \left(f_{n+he_j}+f_{n-he_j}-2f_n\right).
\end{align*}
Here we denote $\mathcal F f=\hat f$. 
Changing the variable, it is further equal to
$$
\frac{1}{h^2}\sum_{n\in h\Z^d}e^{-in\cdot\x}f_n\sum_{j=1}^d \left(e^{ih\x_j}+e^{-ih\x_j}-2\right).
$$
Since 
$$
e^{ih\x_j}+e^{-ih\x_j}-2=-4\sin^2\left(\frac{h\x_j}{2}\right),
$$
we get that 
\begin{align*}
\mathcal F\big[(1-\De_h) f\big](\xi) 
 &=\left[\frac{1}{h^2}\sum_{j=1}^d 4\sin^2\left(\frac{h\x_j}{2}\right)+1\right]\sum_{n\in h\Z^d}e^{-in\cdot\x}f_n\\
 &=\left[\frac{1}{h^2}\sum_{j=1}^d 4\sin^2\left(\frac{h\x_j}{2}\right)+1\right]\hat{f}(\xi).
\end{align*}
Here $f= \{f_n\}$. 
Denote  
$$
M(\x)\triangleq \frac{1}{h^2}\sum_{j=1}^d 4\sin^2\left(\frac{h\x_j}{2}\right)+1=\frac{1}{h^2}\sum_{j=1}^d 2\big(1-\cos h\x_j\big)+1.
$$
Then we obtain that 
\begin{align}\label{def:mult-Delta}
\mathcal F\big[(1-\De_h) f\big](\xi) 
 &=M(\xi)\widehat{f}(\xi).
\end{align}
This implies that $M$ is the multiplier of the operator $(1-\De_h)$. This leads the definition of $(1-\De_h)^{-1}f_n$, which reads as 
\begin{align*}
\mathcal F\big[(1-\De_h)^{-1} f\big](\xi) 
&=m(\xi) \hat f(\xi),\quad \mbox{where } \quad m(\xi)=M(\xi)^{-1}.
\end{align*}

Now we only need to check the condition \eqref{condition-m} in Lemma \ref{lem2.3.1}. 
It is obvious that
$$ 
\|m\|_{L^\infty(I_h)}\le 1.
 $$
By a direct calculation, we have that 
\begin{align}\label{p1-m}
\p_{\x_1}\p_{\x_2}\cdots \p_{\x_d} m(\x)=(-2)^d d! \prod_{j=1}^d \frac{\sin (h\xi_j)} h\> m(\x)^{d+1}.
\end{align}
Note that for any $a>0$, 
\begin{equation}\label{sin-h}
\begin{aligned}
\int_0^{2\pi h^{-1}}\frac{|\sin (h\xi_k)|} h\> m(\x)^{1+a}\,d\x_k
= &
\frac1{h^2}\int_0^{2\pi}\frac{|\sin (\xi_k)|} {\big[h^{-2}\sum_{j=1}^d 2\big(1-\cos \x_j\big)+1\big]^{1+a}}\,d\x_k\\
\le &
\frac1{h^2}\int_0^{2\pi}\frac{|\sin (\xi)|} {\big[h^{-2}\big(1-\cos \x\big)+1\big]^{1+a}}\,d\x\\
\le &
\frac1{h^2}\int_0^1\frac x {\big[h^{-2}x^2+1\big]^{1+a}}\,dx\\
\le &
C \int_0^{+\infty}\frac{x} {(x^2+1)^{1+a}}\,dx\\
\le &
C. 
\end{aligned}
\end{equation}
Applying this estimate and \eqref{p1-m} and choosing $a=\frac1d$, we get that 
\begin{align}\label{est:m-1}
\big\|\p_{\x_1}\p_{\x_2}\cdots \p_{\x_d} m(\x)\big\|_{L^1(I_h)}
\le C .
\end{align}

By \eqref{p1-m} again, we further have that 
\begin{align}\label{p2-m}
\p_{\x_1}^2\p_{\x_2}\cdots \p_{\x_d} m(\x)=(-2)^d d! \left[\cos(h\xi_1)-(d+1) \left(\frac{\sin (h\xi_1)} h\right)^2m(\xi)\right]\prod_{j=2}^d \frac{\sin (h\xi_j)} h\> m(\x)^{d+1}.
\end{align}
Note that 
\begin{subequations}\label{eq10.11.1}
\begin{align}
&\int_0^{2\pi h^{-1}}\left|\cos(h\xi_1)-(d+1) \left(\frac{\sin (h\xi_1)} h\right)^2m(\xi)\right|m(\xi)\,d\x_1\notag\\
&\le 
\int_0^{2\pi h^{-1}}\left|\cos(h\xi_1)\bigg[h^{-2}\sum_{j=1}^d 2\big(1-\cos h\x_j\big)+1\bigg]^{-1}\right|\,d\x_1\label{eq10.11.1.1}\\
&\quad+
C\int_0^{2\pi h^{-1}}\left|\left(\frac{\sin (h\xi_1)} h\right)^2\bigg[h^{-2}\sum_{j=1}^d 2\big(1-\cos h\x_j\big)+1\bigg]^{-2}\right|\,d\x_1. \label{eq10.11.1.2}
\end{align}
\end{subequations}
For \eqref{eq10.11.1.1}, we have
\begin{align*}
\eqref{eq10.11.1.1}
&\leq 
\int_0^{2\pi h^{-1}}\left|\frac{h^{2}}{2\big(1-\cos h\x_1\big)+h^{2}}\right|\,d\x_1 \\
&\xlongequal{\y=h\x_1}
\frac{1}{h} \int_0^{2\pi}\left|\frac{h^{2}}{2\big(1-\cos\y\big)+h^{2}}\right|\,d\y\\
&\leq 
\frac{2}{h}\int_0^{\pi}\frac{h^{2}}{2\big(1-\cos\y\big)+h^{2}}\,d\y.
\end{align*}
Note that for any $0\leq\y\leq \pi$,
\begin{align}\label{eq10.11.3}
\frac{\y^2}{2}\geq 1-\cos\y=2\sin^2 \frac{\y}{2}\geq  \frac{2\y^2}{\pi ^2}.
\end{align}
Using \eqref{eq10.11.3} we have
\begin{align*}
\eqref{eq10.11.1.1}
&\leq 
\frac{2}{h}\int_0^{\pi} \frac{1}{2h^{-2}\y^2+1}\,d\y
\xlongequal{x=h^{-1}\y}
2\int_0^{h^{-1}\pi } \frac{1}{2x^2+1}\,dx 
\leq 
C.
\end{align*}
For \eqref{eq10.11.1.2}
\begin{align*}
\eqref{eq10.11.1.2}
&\leq
\int_0^{2\pi h^{-1}}\left|\left(\frac{\sin (h\xi_1)} h\right)^2\big[2h^{-2} \big(1-\cos h\x_1\big)+1\big]^{-2}\right|\,d\x_1 \\
&\xlongequal{\y=h\x_1}
\frac{2}{h^3}\int_0^{\pi }\frac{\sin ^2\y}{\big[2h^{-2}\big(1-\cos\y \big)+1\big]^{2}}\,d\y
\end{align*}
Applying \eqref{eq10.11.3} again, we have
\begin{align*}
\eqref{eq10.11.1.2}
\leq
\frac{2}{h^3}\int_0^{\pi }\frac{\y^2}{\big(2h^{-2}\y^2+1\big)^{2}}\,d\y
\xlongequal{x=h^{-1}\y}
2\int_0^{h^{-1}\pi  }\frac{x^2}{\big(2x^2+1\big)^{2}}\,d\y
\leq 
C.
\end{align*}
Combining \eqref{eq10.11.1.1},\eqref{eq10.11.1.2} with \eqref{sin-h} gives that 
\begin{align*}
\big\|\p_{\x_1}^2\p_{\x_2}\cdots \p_{\x_d} m(\x)\big\|_{L^1(I_h)}
\le C .
\end{align*}
Similarly, we obtain that 
\begin{align}\label{est:m-2}
\sum\limits_{j=1}^d \big\|\p_{\x_j}\p_{\x_1}\p_{\x_2}\cdots \p_{\x_d} m \big\|_{L^1(I_h)}
\le C .
\end{align}

Now we choose $N=h^{-2}$ in \eqref{condition-m}, and use \eqref{est:m-1} and \eqref{est:m-2}, to obtain that   for any $h\in (0,1]$, 
\begin{align*}
\|m\|_{L^\infty(I_h)}
&+|\ln h|^d \big\|\p_{\x_1}\p_{\x_2}\cdots \p_{\x_d} m\big\|_{L^1(I_h)}\\
&+h |\ln h|^{d-1}\sum\limits_{j=1}^d \big\|\p_{\x_j}\p_{\x_1}\p_{\x_2}\cdots \p_{\x_d} m \big\|_{L^1(I_h)}
\leq C\big(1+|\ln h|\big)^{d}.
\end{align*} 
Therefore, \eqref{condition-m} is valid for $A_0=C\big(1+|\ln h|\big)^{d}$.
Applying Lemma  \ref{lem2.3.1}, we establish the desired estimate and  finish the proof.
\end{proof}

\subsection{Linear operator estimation}

\vskip 0.2cm
In this section, we establish several estimations for linear operators that are essential for our analysis.
\begin{lem}\label{lem2.1}
For any function $f\in l^p_h$, where $1\leq p \leq +\infty$, the following inequality holds:
\begin{align*}
\norm{e^{-it\Delta_h}f}_{l^p_h} \leq e^{2d|1-\frac 2p| h^{-2} t}\norm{f}_{l^p_h}.
\end{align*}
\end{lem}
\begin{proof}
Firstly, we consider the case when $1\leq p < +\infty$.

Let $f=\{f_n\}_{n\in h\Z^d}$, $u_n = e^{-i\Delta_ht}f_n$, then $u_n$ satisfies the equation
\begin{align}\label{eqnonDL}
i\p_t u_n - \Delta_h u_n = 0,
\end{align}
with the initial condition $u_n(0) = f_n$. 
Taking the inner product on both sides of (\ref{eqnonDL}) by $ i u_n$ and changing the variable, we obtain that 
\begin{align*}
\frac12 \p_t\Big(\norm{e^{-it\Delta_h}f}_{l^2_h}^2\Big)
=& \re\langle \p_t u_n,  u_n \rangle\\
= &
 \frac{1}{h^2}\re\Big\langle\sum_{j=1}^d \left(u_{n+he_j}+u_{n-he_j}-2u_n\right), iu_n \Big\rangle \\
 = &
 \frac{1}{h^2}\sum_{j=1}^d\re\big(\langle u_{n+he_j}, iu_n \rangle+\langle u_n, iu_{n+he_j} \rangle\big) \\
= &
0.
\end{align*}
This infers that 
\begin{align}\label{l2-linear}
\norm{e^{-it\Delta_h}f}_{l^2_h} =\norm{f}_{l^2_h}.
\end{align}

By multiplying both sides of (\ref{eqnonDL}) by $\aabs{u_n}^{p-2}\overline{u_n}$, we obtain
\begin{align*}
i\p_t u_n \aabs{u_n}^{p-2} \overline{u_n} = \frac{1}{h^2}\sum_{j=1}^d \left(u_{n+he_j}\aabs{u_n}^{p-2}\overline{u_n}+u_{n-he_j}\abs{u_n}^{p-2}\overline{u_n}-2u_n\aabs{u_n}^{p-2}\overline{u_n}\right).
\end{align*}
Taking the imaginary parts and summing over $n$, we get
\begin{align}\label{eq10.10.1}
\sum_{n\in h\mathbb{Z}^d}\re\left(\partial_t u_n \abs{u_n}^{p-2} \overline{u_n}\right) = \sum_{n\in h\mathbb{Z}^d}\im\left[\frac{1}{h^2}\sum_{j=1}^d \left(u_{n+he_j}\abs{u_n}^{p-2}\overline{u_n}+u_{n-he_j}\abs{u_n}^{p-2}\overline{u_n}\right)\right].
\end{align}
Then, by applying Young's inequality, the above estimate leads to
\begin{align}\label{eq2.a}
\frac{1}{p}\partial_t\sum_{n\in h\mathbb{Z}^d}\left(\abs{u_n}^{p}\right) 
& = \sum_{n\in h\mathbb{Z}^d}\re\left(\partial_t u_n \abs{u_n}^{p-2} \overline{u_n}\right)\notag\\
& \leq \frac{1}{h^2} \sum_{n\in h\mathbb{Z}^d}\sum_{j=1}^d \left(\abs{u_{n+he_j}}\abs{u_n}^{p-1}+\abs{u_{n-he_j}}\abs{u_n}^{p-1}\right)\notag \\
& \leq \frac{1}{h^2}\sum_{j=1}^d \sum_{n\in h\mathbb{Z}^d}\left(\frac{1}{p}\abs{u_{n+he_j}}^{p}+\frac{1}{p}\abs{u_{n-he_j}}^{p}\right)\notag \\
& \quad +\frac{1}{h^2}\sum_{j=1}^d \sum_{n\in h\mathbb{Z}^d}\frac{2p-2}{p}\abs{u_n}^{p}.
\end{align}
Using the change of variable
\begin{align*}
\sum_{n\in h\mathbb{Z}^d}\abs{u_{n+he_j}}^{p} = \sum_{n\in h\mathbb{Z}^d}\abs{u_{n}}^{p},
\end{align*}
equation (\ref{eq2.a}) further implies
\begin{align}\label{eq2.1}
\frac{1}{p}\partial_t\sum_{n\in h\mathbb{Z}^d}\abs{u_n}^{p} \leq \frac{2}{h^2}\sum_{n\in h\mathbb{Z}^d}\sum_{j=1}^d \abs{u_n}^{p} = \frac{2d}{h^2}\sum_{n\in h\mathbb{Z}^d}\abs{u_n}^{p},
\end{align}
which implies
\begin{align*}
\partial_t \sum_{n\in h\mathbb{Z}^d}\abs{u_n}^{p} \leq \frac{2dp}{h^2}\sum_{n\in h\mathbb{Z}^d}\abs{u_n}^{p}.
\end{align*}
Therefore, it gives
\begin{align*}
\sum_{n\in h\mathbb{Z}^d}\abs{u_n(t)}^p \leq  e^{2dpt/h^2} \sum_{n\in h\mathbb{Z}^d}\abs{f_n}^p.
\end{align*}
Consequently, for any $1\leq p < +\infty$, we have
\begin{align}\label{eq2.b}
\left(\sum_{n\in h\mathbb{Z}^d}\abs{u_n(t)}^p\right)^{1/p} \leq e^{2dt/h^2}\left(\sum_{n\in h\mathbb{Z}^d}\abs{f_n}^p\right)^{1/p}.
\end{align}
Notably, this bound is independent of $p$. By taking the limit as $p\rightarrow +\infty$, equation (\ref{eq2.b}) gives
\begin{align*}
\norm{u(t)}_{l^\infty_h} \leq e^{2dt/h^2}\norm{f}_{l^\infty_h}.
\end{align*}
Furthermore interpolation these estimates when $p=1$ and $p=+\infty$ with \eqref{l2-linear}, we obtain the desired estimates. 
This completes the proof of Lemma \ref{lem2.1}.
\end{proof}

Using Corollary \ref{cor1}, we can prove the following lemma, which is the key to establish the well-posedness of the Klein-Gordon equation.
\begin{lemma}\label{lem2.2}
For any $f\in l^p_h$, where $1\leq p\leq+\infty$, the following inequality holds:
$$
\norm{e^{-it\sqrt{1-\Delta_h}}f}_{l^p_h} \leq e^{2\sqrt d|1-\frac 2p| h^{-1}t}\norm{f}_{l^p_h}.
$$
\end{lemma}
\begin{proof}
Let $f = \{f_n\}_{n\in h\Z^d}$, $u_n=e^{-it \sqrt{1-\Delta_h}}f_n$, then $u_n$ satisfies
\begin{align}\label{eqnonKG}
i\partial_t u_n-\sqrt{1-\Delta_h} u_n=0,
\end{align}
with the initial data $u_n(0)=f_n$. 
Taking the inner product on both sides of (\ref{eqnonKG}) by $ i u_n$ and using 
\eqref{def:mult-Delta} and Parseval's identity, 
\begin{align*}
\p_t\Big( \norm{e^{-it\sqrt{1-\Delta_h}}f}_{l^2_h}^2 \Big)
=& 
\re\langle i\p_t u_n, iu_n\rangle \\
=& 
\re\big\langle \sqrt{1-\Delta_h} u_n, iu_n \big\rangle\\
=& 
\re\big\langle M(\xi)^{1/2} \hat u(\xi), i \hat u(\xi)\big \rangle\\
=&0. 
\end{align*}
This gives that 
\begin{align}\label{KG-l2-linear}
\big\|e^{-it\sqrt{1-\Delta_h}}f\big\|_{l^2_h} = \norm{f}_{l^2_h}.
\end{align}

Multiplying both sides of \eqref{eqnonKG} by $\abs{u_n}^{p-2}\overline{u_n}$ and taking the imaginary parts, we obtain
\begin{align*}
\frac{1}{p}\partial_t \sum_{n\in h\mathbb{Z}^d} |u_n|^{p}
 = &
  \sum_{n\in h\mathbb{Z}^d} \re\left(|u_n|^{p-2}\overline{u_n}\partial_t u_n\right) \\
  = &
  \sum_{n\in h\mathbb{Z}^d} \im\left(|u_n|^{p-2}\overline{u_n}\left(1-\Delta_h \right)^{1/2}u_n\right).
\end{align*}
Applying Lemma \ref{lem2.0} and using H\"older's inequality, there  exists  a constant $C$ such that
\begin{align}\label{eq2.a.1}
\frac{1}{p}\partial_t\norm{u}^{p}_{l^{p}_h}&=\sum_{n\in h\mathbb{Z}^d} \im\left(|u_n|^{p-2}\overline{u_n}\left(1-\Delta_h \right)^{1/2}u_n\right) \notag\\
                                      &\leq  \norm{u}^{p-1}_{l^{p}_h}\norm{\left(1-\Delta_h \right)^{1/2}u}_{l^{p}_h}\notag\\
&\leq 2\sqrt d h^{-1} \norm{u}_{l^p_h}^p.
\end{align}
This implies 
\begin{align*}
\norm{e^{-it(1-\Delta)^{1/2}}f}_{l^p_h}\leq e^{2\sqrt d  h^{-1} t}\norm{f}_{l^p_h},\quad 1\leq p\leq \infty.
\end{align*}
Similar as the proof of Lemma \ref{lem2.1}, we use interpolation with \eqref{KG-l2-linear} and obtain the desired estimate. 
This completes the proof of Lemma \ref{lem2.2}.
\end{proof}

\vskip 0.2cm

\subsection{$l^2$-control of the solution to  the DKG}

While we do not possess the conservation law for the $l^2$ norm in DKG, we do have the subsequent $l^2$ norm estimate.
\begin{lemma}\label{lem2.2.3}
Let $h\in (0,1]$, and let $u_n$ be the solution of equation \eqref{defKG} with $\lambda = 1$. Assume that $V_n$ satisfies 
$$
\inf\limits_{n} \big(V_n +2h^{-2}d\big)>0.
$$
 If $(f,g)\in l^2_h\times l^2_h$, then 
 $$
 \norm{(u,\partial_t u)}_{l^2_h\times l^2_h}\leq e^{Ch^{-2}|t|}\norm{(f,g)}_{l^2_h\times l^2_h},
 $$
 where the constant $C>0$ is only dependent of $d, \sigma$ and $\inf\limits_{n} \big(V_n +2h^{-2}d\big)$. 
\end{lemma}
\begin{proof}
By multiplying $\partial_t u_n$ on both sides of equation \eqref{defKG}, we obtain
\begin{align*}
\partial_{tt} u_n(t)\partial_t u_n=\frac1{h^2}\sum_{j=1}^d \left(u_{n+he_j}\partial_t u_n+u_{n-he_j}\partial_t u_n-2u_n\partial_t u_n\right)-V_nu_n \partial_t u_n- |u_n|^{2\sigma}u_n \partial_t u_n.
\end{align*}
Applying Young's inequality and summing over $n$, we have
\begin{align}\label{equation4.1}
&\quad\,\,\,\partial_t\sum_{n\in h\mathbb{Z}^d}\left(\frac{1}{2}\abs{\partial_t u_n}^2+\frac{V_n+2h^{-2}d}{2}\abs{u_n}^2+\frac{1}{2\s+2}\abs{u_n}^{2\s+2}\right)\notag\\
&=\frac1{h^2}\sum_{n\in h\mathbb{Z}^d}\sum_{j=1}^d\left(u_{n+he_j}\partial_t u_n+u_{n-he_j}\partial_t u_n\right) \notag\\
&\leq \frac1{h^2}\sum_{n\in h\mathbb{Z}^d}\sum_{j=1}^d\left(\frac{1}{2}\abs{u_{n+he_j}}^2+\frac{1}{2}\abs{u_{n-he_j}}^2+\abs{\partial_t u_n}^2\right) \notag\\
&=\frac d{h^2}\sum_{n\in h\mathbb{Z}^d} \left(\abs{u_{n}}^2+\abs{\partial_t u_n}^2\right).
\end{align}
This implies
\begin{align*}
\partial_t \sum_{n\in h\mathbb{Z}^d}\left(\abs{\partial_t u_n}^2+\abs{u_n}^2\right)\leq Ch^{-2}\sum_{n\in h\mathbb{Z}^d} \left(\abs{\partial_t u_n}^2+\abs{u_n}^2\right).
\end{align*}
Thus, we conclude the proof of Lemma \ref{lem2.2.3}.
\end{proof}

\vskip 1cm

\section{Well-posedness of discrete nonlinear Schr\"odinger equation}
\vskip 0.3cm
\subsection{Local well-posedness}

In this section, we will establish the local well-posedness of the DNLS. Using Duhamel's formula for the nonlinear Schr\"odinger equation, we define a mapping as follows:
\begin{align}\label{eq3.1}
\Phi(u_n)(t)\triangleq e^{-it\Delta_h}u_{n,0}+i\int_0^t e^{-i(t-s)\Delta_h}\left(V_nu_n+|u_n|^{2\sigma}u_n\right)(s)\,ds.
\end{align}
Moreover denote 
\begin{align}\label{eq10.24.1}
\Phi(u)(t)=\{\Phi(u_n)(t)\}_{n\in h\Z^d},\quad F(u)=\{V_nu_n+|u_n|^{2\s}u_n\}_{n\in h\Z^d}.
\end{align}
Then
$$
\Phi(u)=e^{-it\Delta_h}u_{0}+i\int_0^t e^{-i(t-s)\Delta_h}F(u)(s)ds.
$$
We intend to prove that $\Phi$ defines a contraction mapping on the space $X_R$ defined by
\begin{align*}
X_R\triangleq \left\{u\in C([-T,T]:l^p_h):\|u\|_{L_t^{\infty}{l^p_h}([0,T])}\leq R\right\}.
\end{align*}

Firstly, we claim that $\Phi$ is bounded from $X_R$ to $X_R$. Without loss of generality, we consider the case where $t>0$. Utilizing Lemma \ref{lem2.1}, there exist positive constants $C_1$ and $C_2$ such that
\begin{align}\label{eq3.4}
&\quad\,\,\|{\Phi(u)(t)}\|_{L_t^{\infty}l^p_h([0,T])}\notag\\ 
&\le e^{2d|1-\frac 2p|T/h^2}\|u_{0}\|_{l^p_h}+\|{\int_0^t e^{-i(t-s)\Delta_h}F(u)ds}\|_{L_t^{\infty}l^p_h([0,T])} \notag\\
&\le e^{2d|1-\frac 2p|T/h^2}\|u_0\|_{l^p_h}+\|{\int_0^t \|{e^{-i(t-s)\Delta_h}F(u)}\|_{l^p_h}ds}\|_{L_t^{\infty}([0,T])}\notag\\
&\leq 
e^{2d|1-\frac 2p|T/h^2}\|u_0\|_{l^p_h}+\|\int^t_0 e^{2d|1-\frac 2p|(t-s)/h^2}\,ds\|_{L_t^{\infty}([0,T])}\|F(u)\|_{L_t^{\infty}l^{p}_h([0,T])}\notag\\
&\le e^{2d|1-\frac 2p|T/h^2}\Big(\|u_0\|_{l^p_h}+C_1T\|{V}\|_{l^\infty_h}\|{u}\|_{L_t^{\infty}l^p_h([0,T])}+C_2T\|u\|_{L_t^{\infty}l^{p(2\sigma+1)}_h([0,T])}^{2\sigma+1}\Big). 
\end{align}
Noting that $l^{p(2\sigma+1)}_h\hookrightarrow {l^p_h}$, we further get 
\begin{align}\label{eq3.a.1}
&\|{\Phi(u)(t)}\|_{L_t^{\infty}l^p_h([0,T])}\notag\\ 
\le & e^{2d|1-\frac 2p|T/h^2}\Big(\|u_0\|_{l^p_h}+C_1T\|{V}\|_{l^\infty_h}\|{u}\|_{L_t^{\infty}l^p_h([0,T])}+C_2T\|u\|_{L_t^{\infty}l^{p}_h([0,T])}^{2\sigma+1}\Big). 
\end{align}
Set 
$
R=2\|u_0\|_{l^p_h},
$
then it follows that 
\begin{align*}
\|{\Phi(u)(t)}\|_{L_t^{\infty}l^p_h([0,T])} &\le e^{2dT|1-\frac 2p|/h^2}\big( \frac{1}{2}R+C_1T\|{V}\|_{l^\infty_h}R+C_2TR^{2\sigma+1}\big).
\end{align*}
Choose $T$ suitably small such that
\begin{align*}
e^{2dT|1-\frac 2p|/h^2}\leq \frac{3}{2}, \quad C_1T\|{V}\|_{l^\infty_h}\leq \frac{1}{12}, \quad C_2TR^{2\sigma}\leq \frac{1}{12},
\end{align*}
then we have
\begin{align*}
\|{\Phi(u)(t)}\|_{L_t^{\infty}l^p_h([0,T])}\leq R.
\end{align*}
This implies that $\Phi(u)\in X_R$ for any $u\in X_R$, establishing that $\Phi$ is bounded from $X_R$ to $X_R$.

Next, we proceed to prove that $\Phi$ is a contraction mapping on the space $X_R$. Given $u_n$ and $v_n$ in $X_R$, we have
\begin{align*}
(\Phi(u)-\Phi(v))(t)&=i\int_0^t e^{-i(t-s)\Delta_h}[F(u)-F(v)]ds.
\end{align*}
Employing a similar approach as in (\ref{eq3.4}) and \eqref{eq3.a.1}, there exist positive constants $C_1^{\prime}$ and $C_2^{\prime}$ such that
\begin{align}\label{eq3.6} 
&\quad\,\,\|\Phi(u)-\Phi(v)\|_{L_t^{\infty}l^p_h([0,T])}\notag\\
&\leq
\|\int^t_0 e^{2d|1-\frac 2p|(t-s)/h^2}\,ds\|_{L_t^{\infty}([0,T])}\|F(u)-F(v)\|_{L_t^{\infty}l^{p}_h([0,T])}\notag\\
&\leq 
e^{2d|1-\frac 2p|T/h^2}\bigg[C_1^{\prime}T\|{V}\|_{l^\infty_h}\|{u-v}\|_{L_t^{\infty}l^p_h([0,T])}\notag\\&\quad+C_2^{\prime}T\|u-v\|_{L^\infty_tl^p_h([0,T])}\left(\|{u}\|^{2\sigma}_{L^\infty_tl^p_h([0,T])}+\|{v}\|^{2\sigma}_{L^\infty_tl^p_h([0,T])}\right)\bigg].
\end{align}
Therefore, for any $u, v\in X_R$,  we have that 
\begin{align*}
\|\Phi&(u)-\Phi(v)\|_{L_t^{\infty}l^p_h([0,T])}\\
&\leq  e^{2d|1-\frac 2p|T/h^2}\left(C_1^{\prime}T\|V\|_{l^\infty_h}\|{u-v}\|_{L_t^{\infty}l^p_h([0,T])}
+C_2^{\prime}T R^{2\sigma}\|{u-v}\|_{L_t^{\infty}l^p_h([0,T])}\right).
\end{align*}
With $R=2\|{u_0}\|_{l^p_h}$ and similar to \eqref{eq3.a.1}, choosing sufficiently small $T$, we obtain
\begin{align*}
\|\Phi(u)-\Phi(v)\|_{L_t^{\infty}l^p_h([0,T])}\leq\frac{1}{2}\|{u-v}\|_{L_t^{\infty}l^p_h([0,T])} .
\end{align*}
Hence, we have demonstrated that $\Phi$ is a contraction mapping on space $X_R$, and by the Banach fixed point theorem, we establish the existence and uniqueness of the solution of the equation.

To prove the continuous dependence on the initial data, let $u_n$ and $v_n$ be the corresponding solutions of equation \eqref{defDL} with initial data $u_0$ and $v_0$, respectively. Then we have
\begin{align*}
u_n(t)-v_n(t)&=e^{-it\Delta_h}\left(u_{n,0}-v_{n,0}\right)+i\int_0^t e^{-i(t-s)\Delta_h}V_n\left(u_n-v_n\right)ds\notag\\&\quad+i\int_0^t e^{-i(t-s)\Delta_h}\left(|u_n|^{2\sigma}u_n-|v_n|^{2\sigma}v_n\right)ds.
\end{align*}
Similar to the argument in (\ref{eq3.5}) and (\ref{eq3.6}), we derive
\begin{align}\label{3.7}
\|{u(t)-v(t)}\|_{L_t^{\infty}l^p_h([0,T])}&\lesssim e^{2dT|1-\frac 2p|/h^2}\Big(\|{u_{0}-v_{0}}\|_{l^p_h}+T\|{V}\|_{l^\infty_h}\|{u-v}\|_{L_t^{\infty}l^p_h([0,T])}\notag\\&\quad +T R^{2\sigma}\|{u-v}\|_{L_t^{\infty}l^p_h([0,T])}\Big).
\end{align}
Choosing a sufficiently small time $T$, we establish that
\begin{align*}
\|{u(t)-v(t)}\|_{L_t^{\infty}l^p_h([0,T])}\lesssim\|{u_{0}-v_{0}}\|_{l^p_h},
\end{align*}
which completes the proof of Theorem \ref{thmLOC}.

\subsection{Global well-posedness}
First, we establish the boundedness of the solution:
\begin{lem}\label{lem3.1}
Let $u(t)=\{u_n(t)\}$ be a solution of equation \eqref{defDL}. For any $1\leq p \leq \infty$,
\begin{align*}
\norm{u(t)}_{l^p_h}\leq e^{2dh^{-2}t}\norm{u_0}_{l^p_h}.
\end{align*}
\end{lem}
\begin{proof}
We begin by considering the case when $1\leq p < +\infty$.

Assuming $\{u_n\}$ is a solution, we multiply both sides of the equation by $\aabs{u_n}^p\overline{u_n}$ and obtain 
\begin{align*}
&i\partial_t u_n \aabs{u_n}^p \overline{u_n}-\left(V_n+2dh^{-2}h^{-2}\right)\aabs{u_n}^p\overline{u}_n - \lambda |u_n|^{2\sigma + p + 2}\\
=&\frac1{h^2}\sum^d_{j=1}\left(u_{n + he^j} \aabs{u_n}^p\overline{u_n}+u_{n - he^{j}} \aabs{u_n}^p\overline{u_n}\right).
\end{align*}
Summing over $n$, it gives that 
\begin{align*}
\quad\,\,\sum_{n\in h\Z^d}i\partial_t u_n \aabs{u_n}^p \overline{u_n}
&=
\frac1{h^2}\sum_{n\in h\Z^d}\sum^d_{j=1}(u_{n + he^j} \aabs{u_n}^p\overline{u_n}+u_{n - he^{j}} \aabs{u_n}^p\overline{u_n})\\
&\quad+
\left(V_n+2dh^{-2}h^{-2}\right)|u_n|^{p+2}+\lambda |u_n|^{2\sigma + p + 2}.
\end{align*}
Taking the imaginary parts on both sides, we further get that 
\begin{align*}
\sum_{n\in h\Z^d}\re\left(\partial_t u_n \aabs{u_n}^p \overline{u_n}\right)
&=\frac1{h^2}\im\sum_{n\in h\Z^d}\sum^d_{j=1}\left(u_{n + he^j} \aabs{u_n}^p\overline{u_n}+u_{n - he^{j}} \aabs{u_n}^p\overline{u_n}\right).
\end{align*}
Thus, we arrive at the same estimate as in \eqref{eq10.10.1}. Consequently, we obtain the desired estimate.
\end{proof}

The proof of Theorem $\ref{thm1.3}$ is standard; nevertheless, we present the proof here for the sake of completeness.
\begin{proof}[Proof of Theorem \ref{thm1.3}]
We proceed by contradiction. Let $u_n$ be the local solution obtained through Theorem \ref{thmLOC}. Suppose that there exists a maximal $T_{\ast} < +\infty$. Let $\e$ be a positive constant to be specified later.

Utilizing Duhamel's formula for the nonlinear Schr\"odinger equation, we define a mapping as follows:
\begin{align}\label{eq3.b}
\Phi(u_n)(t)\triangleq e^{-i\De_h (t-t_0)}u_n(t_0)+i\int_{t_0}^t e^{-i\De_h(t-s)}\left(V_nu_n+\aabs{u_n}^{2\s}u_n\right)(s)ds.
\end{align}
With the same notation as \eqref{eq10.24.1}, we write
$$
\Phi(u)(t)=e^{-i\De_h (t-t_0)}u(t_0)+i\int_{t_0}^t e^{-i\De_h(t-s)}F(u)(s)ds.
$$
We aim to show that $\Phi$ constitutes a contraction map on the space $X_{R}$. Here, $X_{R}$ is defined as
\begin{align*}
X_{R}\triangleq \left\{u\in C([T_{\ast}-\e,T_{\ast}-\e+\de]:l^p_h):\norm{u}_{L_t^{\infty}l^p_h([T_{\ast}-\e,T_{\ast}-\e+\de])}\leq R\right\}.
\end{align*}
Our goal is to demonstrate the existence of a sufficiently small constant $\de$ such that $\Phi$ acts as a contraction map on the space $X_{R}$. Consequently, by virtue of the Banach fixed point theorem, the solution also exists within the time interval $[T_{\ast}-\e,T_{\ast}-\e+\de]$. However, this will yield a contradiction, as we will ultimately choose $\e<\de$, which contradicts the maximality of $T_{\ast}$.

To begin, we establish that $\Phi$ maps from $X_{R}$ to $X_{R}$. Accroding to Lemma \ref{lem3.1}, for any $1\leq r \leq +\infty$
\begin{align} \label{eqthm1.3.1}
\norm{u(T_{\ast}-\e)}_{l^p_h}\leq e^{2d(T_{\ast}-\e)/h^2}\norm{u_{0}}_{l^p_h}.
\end{align}
In line with the proof of Theorem $\ref{thmLOC}$, we have
\begin{align}\label{eqthm1.3.1}
&\quad\,\,\norm{\Phi(u)(t)}_{L_t^{\infty}l^p_h([T_{\ast}-\e,T_{\ast}-\e+\de])}\notag\\ 
&\le \norm{e^{-i\De_h (t-T_{\ast}+\e)}u(T_{\ast}-\e)}_{L_t^{\infty}l^p_h([T_{\ast}-\e,T_{\ast}-\e+\de])}\notag\\
&\quad+\normb{\int_{T_{\ast}-\e}^{T_{\ast}-\e+\de} e^{-i\De_h(t-s)}F(u)ds}_{L_t^{\infty}l^p_h([T_{\ast}-\e,T_{\ast}-\e+\de])}. 
\end{align}
Using similar argument in \eqref{eq3.6}, by Lemma \ref{lem2.2} and \eqref{eqthm1.3.1}, and note that $|1-\frac{2}{p}|\leq 1$, there exist constants $C_1$ and $C_2$ such that
\begin{align}
          \eqref{eqthm1.3.1}   &\leq  e^{2d\de/h^2}\norm{u(T_{\ast}-\e)}_{l^p_h}\notag\\
&\quad+\normo{\int_{T_{\ast}-\e}^{T_{\ast}-\e+\de}\normb{ e^{-i\De_h(t-s)}F(u)}_{l^p_h}ds}_{L_t^{\infty}([T_{\ast}-\e,T_{\ast}-\e+\de])}\notag\\
                               &\le e^{2d\de/h^2}\bigg(e^{2d(T_{\ast}-\e)/h^2}\norm{u_{0}}_{l^p_h}+C_1\de\norm{V}_{l^\infty_h}\norm{u}_{L_t^{\infty}l^p_h([T_{\ast}-\e,T_{\ast}-\e+\de])}\notag\\
                               &\quad+C_2\de\norm{\aabs{u}^{2\s}u}_{L_t^{\infty}l^p_h([T_{\ast}-\e,T_{\ast}-\e+\de])}\bigg) \notag\\
                               &\le e^{2d\de/h^2}\left(e^{2d(T_{\ast}-\e)/h^2}\norm{u_{0}}_{l^p_h}+C_1\de\norm{V}_{l^\infty_h}R+C_2\de R^{2\s+1}\right).
\end{align}
By choosing $\de \leq \frac{ln(\frac{3}{2})h^2}{2d}$, we can set
\begin{align*}
\norm{{\Phi(u)(t)}}_{L_t^{\infty}l^p_h([T_{\ast}-\e,T_{\ast}-\e+\de])}
\leq 
\frac{3}{2}e^{2dT_{\ast}/h^2}\norm{u_{0}}_{l^p_h}+\frac{3}{2}\left(C_1\de\norm{V}_{l^\infty_h}R+C_2\de R^{2\s+1}\right).
\end{align*}
Choosing $R=2e^{2dT_{\ast}/h^2}\norm{u_{0}}_{l^p_h}$ and take suitably small $\de$ such that
\begin{align*}
C_1\de\norm{V}_{l^\infty_h}\leq \frac{1}{6}, \quad C_2\de R^{2\s}\leq \frac{1}{6},
\end{align*}
therefore we have
\begin{align*}
\norm{\Phi(u)(t)}_{L_t^{\infty}l^p_h([T_{\ast}-\e,T_{\ast}-\e+\de])}\leq 2 e^{2dpT_{\ast}/h^2}\norm{u_{0}}_{l^p_h}=R,
\end{align*}
indicating that $\Phi$ maps from $X_{R}$ to $X_{R}$.

Now, given $u,v\in X_{R}$, and using an argument analogous to (\ref{eq3.6}), there exist constants $C_1^{\prime}$ and $C_2^{\prime}$ such that
\begin{align}\label{eq3.9}
&\quad\,\,\norm{\Phi(u)-\Phi(v)}_{L_t^{\infty}l^p_h([T_{\ast}-\e,T_{\ast}-\e+\de])}\notag\\
&\leq
\normo{\int_{T_{\ast}-\e}^{T_{\ast}-\e+\de}\normo{ e^{-i\De_h(t-s)}\left(F(u)-F(v)\right)}_{l^p_h}ds}_{L_t^{\infty}([T_{\ast}-\e,T_{\ast}-\e+\de])}\notag\\
&\leq e^{2d\de/h^2}\bigg(C_1^{\prime}\de \norm{V}_{l^\infty_h}\norm{u-v}_{L_t^{\infty}l^p_h([T_{\ast}-\e,T_{\ast}-\e+\de])}\notag\\
&\quad+C_2^{\prime}\de \norm{\left(u-v\right)(|u|^{2\s}+|v|^{2\s})}_{L_t^{\infty}l^p_h([T_{\ast}-\e,T_{\ast}-\e+\de])}\bigg) \notag \\
&\leq e^{2d\de/h^2}\bigg(C_1^{\prime}\de \norm{V}_{l^\infty_h}\norm{u-v}_{L_t^{\infty}l^p_h([T_{\ast}-\e,T_{\ast}-\e+\de])}\notag\\
&\quad+C_2^{\prime}\de  R^{2\s+1}\norm{u-v}_{L_t^{\infty}l^p_h([T_{\ast}-\e,T_{\ast}-\e+\de])}\bigg).
\end{align}
Taking $\de$ such that
\begin{align*}
C_1^{\prime}\de e^{2d\de/h^2}\norm{V}_{l^\infty_h}\leq \frac{1}{4},\quad C_2^{\prime}\de e^{2d\de/h^2} R^{2\s+1}\leq \frac{1}{4},
\end{align*}
then
\begin{align*}
\norm{\Phi(u)-\Phi(v)}_{L_t^{\infty}l^p_h([T_{\ast}-\e,T_{\ast}-\e+\de])}\leq\frac{1}{2}\norm{u-v}_{L_t^{\infty}l^p_h([T_{\ast}-\e,T_{\ast}-\e+\de])} .
\end{align*}
As a result, we have shown that $\Phi$ acts as a contraction map on the space $X_{R}$. By invoking the Banach fixed point theorem, we establish the existence and uniqueness of the solution to the equation within the time interval $[T_{\ast}-\e,T_{\ast}-\e+\de]$. Similarly to the proof of Theorem \ref{thmLOC}, we can verify the continuous dependence of $\Phi(u_n)(t)$ with respect to $u_0$.

Consequently, we conclude that the solution is well-posed within the time interval $[T_{\ast}-\e,T_{\ast}-\e+\de]$. By choosing $\e< \de/2$, we ensure that $T_{\ast}-\e+\de>T_{\ast}$, which contradicts our initial assumption.
\end{proof}

\vskip 0.2cm
\section{Nonlinear discrete Klein-Gordon equation}
\vskip 0.2cm

\subsection{Local well-posedness}
 
In this subsection, we establish the proof for Theorem \ref{thmLOCKG}. Define 
$$\psi_n = (\left(1-\De_h \right)^{-1/2}\p_t - i)u_n,$$ 
which leads to the following system of equations:
\begin{align}
\begin{split}
 \left \{
\begin{array}{ll}
\p_t u_n=\re\left[\sqrt{1-\De_h}\psi_n\right], \\
u_n=-\im(\psi_n).
\end{array}
\right.
\end{split}
\end{align}
This implies that $\psi_n$ satisfies the following equation:
\begin{align*}
\big(\p_t + i\sqrt{1-\De_h}\big)\psi_n=-\left(1-\De_h \right)^{-1/2}\left((V_n-1)u_n+\la|u_n|^{2\sigma}u_n\right).
\end{align*}
By utilizing Duhamel's formula for the wave equation, we define a mapping $\Phi$ as follows:
\begin{align*}
\Phi(\psi_n)(t) &= e^{-i\sqrt{1-\De_h} t}\psi_{n,0} \\
&\quad - \int_0^t e^{-i\sqrt{1-\De_h}(t-s)}\left(1-\De_h \right)^{-1/2}\left(\la|u_n|^{2\sigma}u_n+(V_n-1)u_n\right)ds.
\end{align*}
Moreover, we denote
\begin{align}\label{eq10.24.2}
\Phi(\psi)(t)=\{\Phi(\psi_n)(t)\}_{n\in h\Z^d},\quad F(u)=\{\left(1-\De_h \right)^{-1/2}\left(\la|u_n|^{2\sigma}u_n+(V_n-1)u_n\right)\}_{n\in h\Z^d}.
\end{align}
Then
$$
\Phi(\psi)(t)=e^{-i\sqrt{1-\De_h} t}\psi_{0}- \int_0^t e^{-i\sqrt{1-\De_h}(t-s)}F(u)(s)ds.
$$
Our objective is to demonstrate that $\Phi(\psi_n)$ forms a contraction map on the space $X_R$, defined by:
\begin{align*}
X_R\triangleq \left\{\psi\in C([-T,T]:l^p_h):\norm{\psi}_{L^\infty_tl^p_h([0,T])}\leq R\right\},
\end{align*}
where the constants $T$ and $R$ will be determined. We will focus on the case where $t>0$, as the case for $t<0$ can be proven similarly. Utilizing Corollary \ref{cor1} and Lemma \ref{lem2.2}, there exist positive constants $C_1,...,C_5$ dependent of $d,h$ and $\norm{V}_{l^{\infty}_h}$ such that:
\begin{align}\label{eq4.1}
&\quad\,\,\norm{\Phi(\psi)(t)}_{L_t^{\infty}l^p_h([0,T])}\notag\\
&\le 
\norm{e^{-i\sqrt{1-\De_h} t}\psi_0}_{L_t^{\infty}l^p_h([0,T])}+\normo{\int_0^t e^{-i\sqrt{1-\De_h}(t-s)}F(u)\,ds}_{L_t^{\infty}l^p_h([0,T])} \notag\\
&\leq
e^{2\sqrt d|1-\frac 2p| h^{-1}T}\norm{\psi_0}_{l^p_h}+\normo{\int_0^t \norm{e^{-i\sqrt{1-\De_h}(t-s)}F(u)}_{l^p_h}\,ds}_{L_t^{\infty}([0,T])}\notag\\
&\leq 
e^{2\sqrt d|1-\frac 2p| h^{-1}T}\bigg(\norm{\psi_0}_{l^p_h}+C_1T\normo{ \left(1-\De_h \right)^{-1/2}\left(\la\aabs{u}^{2\s}u\right)}_{L_t^{\infty}l^p_h([0,T])}\notag\\
&\quad +
C_2T\norm{\left(1-\De_h \right)^{-1/2}(V-1)}_{l^\infty_h}\norm{u}_{L_t^{\infty}l^p_h([0,T])} \bigg)\notag\\
&\leq 
e^{2\sqrt d|1-\frac 2p| h^{-1}T}\left(\norm{\psi_0}_{l^p_h}+C_3T\norm{ \psi}^{2\s+1}_{L_t^{\infty}{l^p_h}([0,T])} +C_4T\norm{ \psi}_{L_t^{\infty}l^p_h([0,T])}\right) \notag\\
&\leq 
e^{2\sqrt d|1-\frac 2p| h^{-1}T}\left(\norm{\psi_0}_{l^p_h}+C_3TR^{2\s+1}+C_4TR\right).
\end{align}
By selecting $R=2\norm{\psi_0}_{l^p_h}$, similar to \eqref{eq3.a.1}, taking suitably small $T$, we ensure that $\norm{\Phi(\psi)(t)}_{L_t^{\infty}l^p_h([0,T])}< 2\norm{\psi_0}_{l^p_h}$, thus demonstrating that $\Phi$ is bounded from $X_R$ to $X_R$. 

Now, for $\psi,\tilde{\psi}\in X_R$, employing a similar argument as in (\ref{eq4.1}), we can find constants $C_1^{\prime},C_2^{\prime}$ that satisfy the following:
\begin{align}\label{equation4.3}
&\quad\,\,\norm{(\Phi(\psi)-\Phi(\tilde{\psi}))(t)}_{L_t^{\infty}l^p_h([0,T])}\notag\\
&\leq e^{2\sqrt d|1-\frac 2p| h^{-1}T}\bigg(C_1^{\prime}T\norm{\left(\psi-\tilde{\psi}\right)\left(|\psi|^{2\s}+|\tilde{\psi}|^{2\s}\right)}_{L_t^{\infty}l^p_h([0,T])}\notag\\
&\quad+
C_2^{\prime}T\norm{\psi-\tilde{\psi}}_{L_t^{\infty}l^p_h([
0,T])}\bigg) \notag \\
&\leq e^{2\sqrt d|1-\frac 2p| h^{-1}T}\left(C_1^{\prime}T R^{2\s}+C_2^{\prime}T\right)\norm{\psi-\tilde{\psi}}_{L_t^{\infty}l^p_h([0,T])}.
\end{align}
By selecting $R=2\norm{\psi_{0}}_{l^p_h}$ and taking suitably small $T$, we obtain
\begin{align*}
\norm{(\Phi(\psi)-\Phi(\tilde{\psi}))(t)}_{L_t^{\infty}l^p_h([0,T])}\leq\frac{1}{2}\norm{\psi-\tilde{\psi}}_{L_t^{\infty}l^p_h([0,T])} .
\end{align*}
Hence, $\phi$ is a contraction map on $X_R$, and applying the Banach fixed point theorem, we establish the existence and uniqueness of the solution to the equation.

To prove the continuous dependence of $\phi(\psi)(t)$ with respect to $\psi_0$, let $u$ and $v$ be solutions of equation \eqref{defDL} with initial data $\psi_0$ and $\tilde{\psi}_0$, respectively. Similar to the argument in (\ref{eq3.5}) and (\ref{eq3.6}), we deduce:
\begin{align*}
\norm{\psi(t)-\tilde{\psi}(t)}_{L_t^{\infty}l^p_h([0,T])}
&\lesssim e^{2\sqrt d|1-\frac 2p| h^{-1}T}\norm{\psi_0-\tilde{\psi}_{0}}_{l^p_h}\\
&\quad +
e^{2\sqrt d|1-\frac 2p| h^{-1}T}\left( R^{2\s}+1\right)T\norm{\psi-\tilde{\psi}}_{L_t^{\infty}l^p_h([0,T])}.
\end{align*}
By following analogous steps and selecting $T$ to be sufficiently small, we conclude that
\begin{align*}
\norm{\psi(t)-\tilde{\psi}(t)}_{L_t^{\infty}l^p_h([0,T])}\lesssim\norm{\psi_0-\tilde{\psi}_{0}}_{l^p_h},
\end{align*}
which verifies the continuous dependence of $\phi(\psi)$ and completes the proof of Theorem \ref{thmLOCKG}.

\subsection{Global well-posedness in the defocusing case}

Firstly we prove the of boundedness of $u_n$.

\begin{lem}\label{lem4.0}
Let $d\geq 1$, $1\leq p\leq +\infty$, $0<h\leq 1$, $v=\{v_n\}_{n\in h\Z^d}$ is real-valued sequence.  Assume that $V=\{V_n\}_{n\in h\Z^d}$ satisfing
$$
\norm{V}_{\l^\infty_h}<\infty.
$$ 
Denote $(f,g)=\{(f_n,g_n)\}_{n\in h\Z^d}$, 
suppose  that $(f,g)\in l^p_h\times l^p_h$, and let $v_n$ be the solution of
\begin{align}\label{eqthm4.1.1}
\begin{split}
 \left \{
\begin{array}{ll}
\p_{tt}v_{n}-\Delta_hv_n+V_nv_n =0,\\
v_{n,0}=f_n, \quad\p_t v_n(0)=g_n,
\end{array}
\right.
\end{split}
\end{align}
then
\begin{align*}
\norm{(v,\p_t v)}_{l^p_h\times l^p_h}\le  Ch^{-1}e^{Ch^{-1}t}\norm{(f,g)}_{l^p_h\times l^p_h},
\end{align*}
where the constant $C$ only dependent of $d$ and $\norm{V}_{l^\infty_h}$.
\end{lem}
\begin{proof}
Let $\psi_n=\left[(1-\De_h)^{-1/2}\p_t-i\right]v_n$, then $\psi_n$ follows the equation
\begin{align}\label{eqlem4.1.1}
\p_t\psi_n+i\sqrt{1-\De_h}\psi_n=(1-\De_h)^{-1/2}\big[(1-V_n)v_n\big],
\end{align}
with the initial data 
\begin{align}\label{intial-psi}
\psi_n(0)=\psi_{n,0}\triangleq (1-\De_h)^{-1/2}g_n-if_n.
\end{align}
It follows that
\begin{align}\label{eqlem4.1.2}
\begin{split}
 \left \{
\begin{array}{ll}
\p_t v_n=\re\left[\sqrt{1-\De_h}\psi_n\right], \\
v_n=-\im(\psi_n).
\end{array}
\right.
\end{split}
\end{align}
By \eqref{eqlem4.1.1} and Duhamel's formula, we have that 
\begin{align*}
\psi_n(t)=e^{-it\sqrt{1-\De_h}}\psi_{n,0}+\int_0^t e^{-i(t-s)\sqrt{1-\De_h}}(1-\De_h)^{-1/2}\big[(1-V_n)v_n(s)\big]\,ds.
\end{align*}
Similar as section 4.1, we denote
\begin{align}\label{eq10.24.3}
\psi(t)=\{\psi_n(t)\}_{n\in h\Z^d},\quad F(v)=\{(1-\De_h)^{-1/2}\big[(1-V_n)v_n\big]\}_{n\in h\Z^d}.
\end{align}
Then
$$
\psi(t)=e^{-it\sqrt{1-\De_h}}\psi_{0}+\int_0^t e^{-i(t-s)\sqrt{1-\De_h}}F(v)(s)\,ds.
$$
Then by Corollary \ref{cor1} and Lemma \ref{lem2.2}, 
\begin{align}\label{eqlem4.1.3}
\norm{\psi(t)}_{l^p_h}
&\leq 
\norm{e^{-it\sqrt{1-\De_h}}\psi_0}_{l^p_h}+\normo{\int_0^t e^{-i(t-s)\sqrt{1-\De_h}}F(v)(s)\,ds}_{l^p_h}\notag\\
&\leq 
e^{2\sqrt{d}h^{-1}t}\norm{\psi_0}_{l^p_h}+\int_0^t \normo{e^{-i(t-s)\sqrt{1-\De_h}}F(v)}_{l^p_h}\,ds\notag\\
&\leq
e^{2\sqrt{d}h^{-1}t}\norm{\psi_0}_{l^p_h}+\int_0^t e^{2\sqrt{d}h^{-1}(t-s)}\big(1+|\ln h|\big)^{d/2}\left(1+\norm{V}_{l^{\infty}_h}\right)\normo{v(s)}_{l^p_h}\,ds\notag\\
&\leq
e^{2\sqrt{d}h^{-1}t}\norm{\psi_0}_{l^p_h}+\big(1+|\ln h|\big)^{d/2}\left(1+\norm{V}_{l^{\infty}_h}\right)\int_0^t e^{2\sqrt{d}h^{-1}(t-s)}\normo{\psi(s)}_{l^p_h}\,ds.
\end{align}
Let 
$$
G(t)\triangleq e^{-2\sqrt{d}h^{-1}t}\norm{\psi(t)}_{l^p_h}.
$$
From \eqref{eqlem4.1.3}, we can see that
$$
G(t)\leq G(0)+\int^t_0 \big(1+|\ln h|\big)^{d/2}\left(1+\norm{V}_{l^{\infty}_h}\right)G(s)ds.
$$
Then by Gronwall inequality, we have
$$
G(t)\leq G(0)e^{\big(1+|\ln h|\big)^{d/2}\left(1+\norm{V}_{l^{\infty}_h}\right)t}.
$$
Hence, there exists a constant $C_0$ dependent of $d$ and $\norm{V}_{l^{\infty}_h}$ such that 
\begin{align}\label{eqlem4.1.4}
\norm{\psi(t)}_{l^p_h}\leq & e^{ C_0(h^{-1}+|\ln h|^{d/2})t}\norm{\psi_0}_{l^p_h}\notag\\
\leq & e^{C_0 h^{-1}t}\norm{\psi_0}_{l^p_h}.
\end{align}
Then by \eqref{eqlem4.1.1} and \eqref{eqlem4.1.2}
\begin{align}\label{eqlem4.1.5}
\norm{v(t)}_{l^p_h}
&=
\norm{\im{\psi(t)}}_{l^p_h}\notag\\
&\leq
e^{C_0h^{-1}t}\norm{\psi_0}_{l^p_h}\notag\\
&\leq
e^{C_0h^{-1}t}\norm{(1-\De_h)^{-1/2}g_n-if_n}_{l^p_h}\notag\\
&\leq
e^{C_0h^{-1}t}\left[C_1(1+|\ln h|^{d/2})\norm{g}_{l^p_h}+\norm{f}_{l^p_h}\right],
\end{align}
where $C_1$ is a positive constant only dependent of $d$.

Similar to \eqref{eqlem4.1.5}, we have
\begin{align}\label{eqlem4.1.6}
\norm{\p_t v(t)}_{l^p_h}
&=
\norm{\re{\sqrt{1-\De_h}\psi(t)}}_{l^p_h}\notag\\
&\leq
e^{C_0h^{-1}t}\norm{\sqrt{1-\De_h}\psi_0}_{l^p_h}\notag\\
&\leq
e^{C_0h^{-1}t}\norm{g_n-i\sqrt{1-\De_h}f_n}_{l^p_h}\notag\\
&\leq
e^{C_0h^{-1}t}\left(\norm{g}_{l^p_h}+C_2h^{-1}\norm{f}_{l^p_h}\right),
\end{align}
where $C_2$ is a positive constant  only dependent of $d$. Combining with \eqref{eqlem4.1.5} and \eqref{eqlem4.1.6}, we get the desired result.
\end{proof}

Followed from \eqref{assump-Vn-KG-1}, there exists some constant $\delta_0>0$ such that 
\begin{align}\label{assump-Vn-KG-1-1}
\inf_{n\in h\Z^d}\{V_n\}+(2d-\delta_0)h^{-2}\ge 0.
\end{align} 
Then we have the following proposition.
\begin{prop}\label{thm4.1}
Let  $d,\s, h,\lambda$ and $ p$ be under the same assumptions with Theorem 1.5,   $V_n$ satisfy \eqref{assump-Vn-KG-1-1}, and $u_n$ be the solution of \eqref{defKG}.
Suppose that $(f,g)\in l^p_h\times l^p_h$, then
\begin{align*}
\norm{\big(u,\p_tu\big)}_{l^p_h\times l^p_h}\leq C\delta_0^{-1}h^{-1}e^{C_0\delta_0^{-1}h^{-1}t},
\end{align*}
where the positive constants $C=C(d,\s, \norm{V}_{l^\infty_h}, \norm{(f,g)}_{l^p_h\times l^p_h})$ and $C_0=C_0(d,\s, \norm{V}_{l^\infty_h})$.
\end{prop}
\begin{proof}
Let $v_n$ be the solution of the following equation
\begin{align}
\begin{split}
 \left \{
\begin{array}{ll}
\p_{tt}v_{n}-\Delta_hv_n+V_nv_n =0\\
v_{n}(0)=f_n, \quad\p_t v_{n}(0)=g_n,
\end{array}
\right.
\end{split}
\end{align}
Let $w_n=u_n-v_n$, then $w_n$ obeys the following equation
\begin{align}\label{eqthm4.2.1}
\begin{split}
 \left \{
\begin{array}{ll}
\p_{tt}w_n-\Delta_hw_n+V_nw_n =-|u_n|^{2\s}u_n,\\
w_{n}(0)=0, \quad\p_t w_{n}(0)=0.
\end{array}
\right.
\end{split}
\end{align}

Mutiplying both sides of equation (\ref{eqthm4.2.1}) by $\p_t w_n$ and sum by $n$, we get
\begin{align}\label{eqthm4.2.2}
 \sum_{n\in h\Z^d}\p_{tt}w _n\p_t w_n& =\sum_{n\in h\Z^d}\left[\frac{1}{h^2}\sum_{j=1}^d\left(w_{n+he_j}+w_{n-he_j}\right)\p_t w_n\right]-\sum_{n\in h\Z^d}(V_n+2dh^{-2})w_n\p_tw _n\notag \\
      & \quad-\sum_{n\in h\Z^d}|u_n|^{2\sigma}u_n\left(\p_tu_n-\p_t v_n\right).
 \end{align}
Define the modified energy 
\begin{align*}
E(t)\triangleq \sum_{n\in h\Z^d}\left(\frac{(V_n+2dh^{-2})|w_n|^2}{2}+\frac{|\p_tw_n|^2}{2}+\frac{|u_n|^{2\s+2}}{2\s+2}\right).
\end{align*}
By (\ref{eqthm4.2.2}), we have
\begin{align*}
\p_tE(t)
&=
 \sum_{n\in h\Z^d}\left[\frac{1}{h^2}\sum_{j=1}^d\left(w_{n+he_j}+w_{n-he_j}\right)\p_t w_n\right]+\sum_{n\in h\Z^d}|u_n|^{2\s}u_n\p_t v_n \\
&\leq
\frac{2d}{\delta_0 h}\sum_{n\in h\Z^d}\left(\frac{\delta_0}{2h^2}|w_n|^2+\frac12|\p_tw_n|^2\right)+\frac1{2\s+2}\sum_{n\in h\Z^d}|u_n|^{2\s+2}\\
&\qquad +\frac{2\s+1}{2\s+2}\sum_{n\in h\Z^d}|\p_tv_n|^{2\s+2}.
\end{align*}
Note that 
$V_n+2dh^{-2}\ge \delta_0h^{-2}$, 
we further get that 
\begin{align}\label{eqlem4.2.1}
\p_tE(t)
&\leq
\frac{2d}{\delta_0 h}E(t)+\frac{2\s+1}{2\s+2}\sum_{n\in h\Z^d}|\p_tv_n|^{2\s+2}.
\end{align}
By Lemma \ref{lem4.0}, 
\begin{align}\label{eqlem4.2.2}
\frac{2\s+1}{2\s+2}\sum_{n\in h\Z^d}|\p_tv_n|^{2\s+2}
&\leq 
C_1h^{-1}e^{C_1h^{-1}t}\norm{(f,g)}_{l^{2\s+2}_h\times l^{2\s+2}_h}^{2\s+2}.
\end{align}
Here and below in this proof, denote $C_j$, where $j = 1, 2, ..., $ as positive constants that are independent of $d$, $\sigma$, and $\norm{V}_{l^{\infty}_h}$.

Substituting \eqref{eqlem4.2.2} into \eqref{eqlem4.2.1}, and noting that $p\leq 2\s+2$, we have 
\begin{align*}
\p_t E(t)\leq  \frac{2d}{\delta_0 h}{E}(t)+C_1h^{-1}e^{C_1h^{-1}t}\norm{(f,g)}_{l^p_h\times l^p_h}^{2\s+2},
\end{align*}
where $C_2$ is a constant dependent of $d,\s$ and $\norm{V}_{l^{\infty}_h}$.

Therefore,
\begin{align*}
E(t)&\leq e^{2d\delta_0^{-1}h^{-1}t}E(0)+C_3h^{-1}e^{C_4\delta_0^{-1}h^{-1}t}\norm{(f,g)}_{l^p_h\times l^p_h}^{2\s+2}.
\end{align*}
Note that 
\begin{align*}
E(0)
=&\sum_{n\in h\Z^d}\left(\frac{(V_n+2dh^{-2})|w_{n}(0)|^2}{2}+\frac{|\p_tw_n(0)|^2}{2}+\frac{|u_{n}(0)|^{2\s+2}}{2\s+2}\right)\\
=&\frac1{2\s+2}\sum_{n\in h\Z^d}|f_{n}|^{2\s+2}.
\end{align*}
This further gives that 
\begin{align*}
E(t)&\leq 
C_5h^{-1}e^{C_6h^{-1}t}\norm{(f,g)}_{l^p_h\times l^p_h}^{2\s+2},
\end{align*}
Since  
\begin{align*}
\norm{\big(w,\p_tw\big)}_{l^2_h\times l^2_h}
\leq & 2\delta_0^{-1}E(t ),
\end{align*}
we obtain that 
\begin{align*}
\norm{\big(w,\p_tw\big)}_{l^2_h\times l^2_h}
\leq &2C_5\delta_0^{-1} h^{-1}e^{C_6\delta_0^{-1}h^{-1}t}\norm{(f,g)}_{l^p_h\times l^p_h}^{2\s+2}.
\end{align*}
Therefore,
\begin{align*}
\norm{\big(u,\p_tu\big)}_{l^p_h\times l^p_h}
\leq & \norm{\big(w,\p_tw\big)}_{l^2_h\times l^2_h}+\norm{\big(v,\p_tv\big)}_{l^p_h\times l^p_h}\\
\leq & \norm{\big(w,\p_tw\big)}_{l^p_h\times l^p_h}+\norm{\big(v,\p_tv\big)}_{l^p_h\times l^p_h}\\
\leq & C\delta_0^{-1}h^{-1}e^{C_0\delta_0^{-1}h^{-1}t},
\end{align*}
where the constants $C=C(d,\s, \norm{V}_{l^\infty_h}, \norm{(f,g)}_{l^p_h\times l^p_h})>0,C_0=C_0(d,\s, \norm{V}_{l^\infty_h})>0$. 
This finishes the proof of Proposition \ref{thm4.1}.
\end{proof}
Theorem \ref{thm1.4} follows from a  similar argument in Theorem \ref{thm1.3}.

\subsection{Blowing-up in the focusing case}\label{sec:blowup}
\vskip 0.2cm
For the focusing case, we prove blow-up for the solution.

Let $u=\{u_n\}_{n\in h\Z^d}$ be a solution to the Cauchy problem  \eqref{defKG} with initial data  $(f,g)=\{(f_n,g_n)\}_{n\in h\Z^d}$, define
\begin{align}\label{eq4.17.1}
I(t)=\sum_{n\in h\Z^d} u_n^2(t).
\end{align}
The key to prove Theorem \ref{mainthm3} is the following estimation of $I^{\prime\prime}(t)$.
\begin{lem} \label{lem2} Under the same assumptions of $d,\la$ and $V_n$ with Theorem \ref{mainthm3}, let $(f,g)$ be the initial data such that $E(f,g)<0$, then the corresponding solution $u_n$ of $(f,g)$ satisfing
\begin{align*}
I^{\prime\prime}(t)\geq (4+2\s)\sum_n \aabs{\p_t u_n}^2
\end{align*}
\end{lem}
\begin{proof}
By \eqref{eq4.17.1}, we have
\begin{align*}
I^{\prime\prime}(t)&=2\sum_{n\in h\Z^d} \aabs{\p_t u_n}^2 + 2\sum_{n\in h\Z^d} u_n\p_{tt}u_n \\
                    &=2\sum_{n\in h\Z^d} \aabs{\p_t u_n}^2 + 2\sum_{n\in h\Z^d} u_n\left(\De_h u_n-V_nu_n+ |u_n|^{2\sigma}u_n\right) \\
                    &=2\sum_{n\in h\Z^d} \aabs{\p_t u_n}^2 + 2\sum_n u_n\left(\frac{1}{h^2}\sum_{j=1}^d\left(u_{n+he_j}+u_{n-he_j}-2u_n\right)-V_nu_n+|u_n|^{2\sigma}u_n\right).
\end{align*}
By variable substitution
\begin{align}\label{eq4.0}
I^{\prime\prime}(t)
                    &=2\sum_{n\in h\Z^d} \aabs{\p_t u_n}^2 + 2\sum_{n\in h\Z^d}\left(-V_n|u_n|^2+|u_n|^{2\s+2}\right)\notag\\
&\quad+\frac{2}{h^2}\sum_{j=1}^d\sum_{n\in h\Z^d}\left(2u_nu_{n+he_j}-|u_{n+he_j}|^2-|u_n|^2\right)\notag \\
                    &=2\sum_{n\in h\Z^d} \aabs{\p_t u_n}^2 + 2\sum_{n\in h\Z^d}\left(-V_n|u_n|^2+|u_n|^{2\s+2}\right)\notag\\&\quad-\frac{2}{h^2}\sum_{n\in h\Z^d}\sum_{j=1}^d\left(u_{n+he_j}-u_n\right)^2.
\end{align}
On the other hand, energy conservation law shows that
\begin{align}\label{eq4.2}
&\quad\sum_{n\in h\Z^d} |u_n|^{2\s+2}+\left(2\s + 2\right)E(f,g)\notag \\&\equiv  \left(\s + 1\right)\sum_{n\in h\Z^d} \left(\frac{1}{h^2}\sum_{j=1}^d\left(u_{n+he_j}-u_n\right)^2+V_n|u_n|^2+\aabs{\p_t u_n}^2\right).
\end{align}
Substituting (\ref{eq4.2}) into (\ref{eq4.0}), we get
\begin{align*}
I^{\prime\prime}(t)&= (4 + 2\s) \sum_{n\in h\Z^d}\aabs{\p_t u_n}^2 + 2\s\sum_{n\in h\Z^d}\left(\frac{1}{h^2}\sum_{j=1}^d\left(u_{n+he_j}-u_n\right)^2+V_n|u_n|^2\right) \\&\quad- (4 + 4\s)E(f,g).
\end{align*}
which implys
\begin{align*}
I(t)^{\prime\prime}\geq \sum_{n\in h\Z^d} (4+2\s)\aabs{\p_t u_n}^2  - (4 + 4\s)E(f,g) > 0.
\end{align*}
This finishes the proof of Lemma \ref{lem2}.
\end{proof}

Now we are ready to prove Theorem \ref{mainthm3}.
\begin{proof}
We prove by contradiction. Assuming $T=\infty$ be the maximal lifespan of $u_n$. By Lemma \ref{lem2}, for any solution $u_n$ with initial data $(f,g)$ satisfing $E(f,g)<0$, we have $$I^{\prime\prime}(t)> (4+2\s)\sum_{n\in h\Z^d}\aabs{\p_t u_n}^2>0,$$ for any $t\in[0,\infty)$. Then there exists $t_1\in (0,\infty)$ such that $I^{\prime}(t)>0$ and $I(t)>0$ for any $t\in [t_1,\infty)$. Then by Lemma \ref{lem2}
\begin{align*}
\quad\quad I^{\prime\prime}(t)I(t)&-(\s/2 + 1)I^{\prime}(t)^2 \\
&\geq 
\left(\sum_n \aabs{u_n}^2\right)\left((4+2\s)\sum_n \aabs{\p_t u_n}^2\right) - 4(1+ \s/2)\left(\sum_n u_n\p_t u_n\right)^2 \\
&\geq 
(4+2\s)\left(\sum_n \aabs{u_n}^2\right)\left(\sum_n \aabs{\p_t u_n}^2\right) - (4+2\s)\left(\sum_n u_n\p_t u_n\right)^2 \\
&
\geq 0.
\end{align*}
Thus, for $t\in [t_1, \infty)$, we have
\begin{align*}
&\left(I(t)^{-\s/2}\right)^{\prime}=-\frac{\s}{2} I(t)^{-\s/2 - 1} I^{\prime}(t) <0,\\
&\left(I(t)^{-\s/2}\right)^{\prime\prime}=-\frac{\s}{2} I(t)^{-\s/2-2}\left[I^{\prime\prime}(t)I(t)-(\s/2 + 1)I^{\prime}(t)^2\right]\leq 0.
\end{align*}
Therefore, 
\begin{align*}
I(t)^{-\s/2}\leq I(t_1)^{-\s/2} -\frac{\s}{2} I(t_1)^{-\s/2-1}I^{\prime}(t_1)(t-t_1),\quad t\in [t_1,\infty).
\end{align*}
So there exists $t_2\in [t_1,\infty)$ such that $I(t_2)^{-\s/2}<0$, this contradicts \eqref{eq4.17.1}.
\end{proof}

\bibliographystyle{plain}
\bibliography{reference}

\end{document}